\documentclass{amsart}
\usepackage{xcolor}

\usepackage{pdfsync}

\usepackage{lscape}
\usepackage{graphicx}
\usepackage{epstopdf}    
\usepackage[all]{xy}
\usepackage{amsmath,amsthm,amssymb,enumerate}
\usepackage{latexsym}
\usepackage{amscd}
\usepackage{amssymb}

\numberwithin{equation}{section}

\theoremstyle{plain}  
\newtheorem{thm}[equation]{Theorem}

\newtheorem{lemma}[equation]{Lemma}

\theoremstyle{definition}  
\newtheorem{defn}[equation]{Definition}

\newtheorem{remark}[equation]{Remark}

\newtheorem{fact}[equation]{Fact}

\newcommand{\ra}{\rightarrow}

\newcommand{\lra}{\longrightarrow}

\newcommand{\R}{\mathbb R}
\newcommand{\PP}{\mathbb P}
\newcommand{\pontr}{\hat{p}_1}

\newcommand{\Z}{\mathbb Z}
\newcommand{\C}{\mathbb C}

\newcommand{\cp}{\C P^\infty}

\newcommand{\Zp}{\Z_{(2)}}
\newcommand{\Zq}{\Z/(2)}
\newcommand{\Zqq}{\Z/(2^q)}

\newcommand{\uhat}{\hat{u}}
\newcommand{\vhat}{\hat{v}}

 \newcommand{\Smash} {\wedge}

\thanks{
The third author would like to thank the Academy of Mathematics and Systems
Science, part of the Chinese Academy of Sciences, for its hospitality and
support during part of the research for this paper.
}

\thanks{
The authors thank the referee for numerous suggestions to improve the exposition.
}

\thanks{
The first author is supported in part by the NSF through grant DMS 1307875.
}


\begin{document}

\title{The $ER(2)$-cohomology of $B\Zqq$ and $\C \PP^n$}

\author{Nitu Kitchloo}
\author{Vitaly Lorman}
\author{W. Stephen Wilson}
\address{Department of Mathematics, Johns Hopkins University, Baltimore, USA}
\email{nitu@math.jhu.edu}
\email{vlorman@math.jhu.edu}
\email{wsw@math.jhu.edu}

\begin{abstract}
The $ER(2)$-cohomology of $B\Zqq$ and $\C \PP^n$ are computed along with
the Atiyah-Hirzebruch spectral sequence for $ER(2)^*(\C \PP^\infty)$.
This, along with other papers in this series, gives 
us the $ER(2)$-cohomology of all Eilenberg-MacLane spaces. 
\end{abstract}

\date{\today}

\maketitle

\section{Introduction}

We will be concerned with only two cohomology theories in this paper.
All our work is at the prime 2.
First, we have the Johnson-Wilson theory, $E(2)^*(-)$, 
introduced in \cite[Remark 5.13]{JW2}, 
with coefficients $E(2)^* $ $=  \Zp [v_1 , v_2^{\pm 1}]$
where the degree of $v_1$ is -2 and the degree of $v_2$ is -6.

Second, 
the Real Johnson-Wilson theory, 
$ER(2)^*(-)$, from 
\cite[Theorem 1.7]{Nitufib} and \cite[Theorem 4.1]{HK},
is
the main theory of interest. 
The theory $E(2)$ is complex orientable and it inherits 
a $\Zq$-action from complex conjugation on $MU$, the spectrum for complex cobordism.
The theory 
$ER(2)$ is
the homotopy fixed points of the spectrum
$E(2)$ under this action and is
just the $n=2$ analog of $ER(1) = KO_{(2)}$.

This paper is part of a series of papers developing the generalized
cohomology theory, $ER(2)^*(-)$, (and often $ER(n)^*(-)$), as a working tool
for algebraic topologists.  Interest in this comes from two 
directions.  First, there is a close connection between $ER(2)$ and
$TMF_0(3)$ (see \cite[Corollary 4.17]{HM}), and second, $ER(2)$ has already 
proven useful in applications,
particularly to non-immersions of real projective spaces, for example 
in \cite[Theorem 1.9]{NituP},
\cite[Theorem 1.4]{NituP2}, and \cite[Theorem 4.1]{Romie}.

A great deal is known about $ER(2)$ already. In particular, we know
the homology of the Omega spectrum for $ER(2)$, \cite[Theorem 1.2 and
Section 2]{NituER2}, and
the homotopy type of the spaces in the Omega spectrum, 
\cite[Theorems 1-4 and 1-6 and related discussion]{Kitch-Wil-split}.
For most $n$, $ER(2)^*(\R\PP^n)$ has been computed,
\cite[Theorems 13.2 and 13.3 for $n$ even]{NituP},
\cite[Theorem 8.2 for $n=16k+1$]{NituP2}, and \cite[Theorem 3.1 for $n=16k+9$]{Romie}.
We also know $ER(n)^*(BO(q))$, 
\cite[Theorem 1.1]{Kitch-Wil-BO}.

It is hard to put the results of this paper into proper context
because the context is constantly expanding.  
This paper is part of a much larger project of developing the
computability and applicability of $ER(n)^*(-)$.  
Computing
\begin{equation*}
ER(n)^*(K(\Zqq,j)) \text{ and }
ER(n)^*(K(\Z,j+1)) 
\end{equation*}
are long term goals.
In \cite[Theorem 1.3]{Kitch-Lor-Wil-ERn}, the $ER(n)$-cohomology is computed
for 
\begin{equation*}
K(\Z,2k+1), K(\Zqq,2k), \text{ and } K(\Zq,2k+1).
\end{equation*}
This paper can be seen as the first attack on the odd $K(\Zqq,2k+1)$
cases, which are considerably more complicated than the even ones.

In the case of $ER(2)$, we know 
\begin{equation*}
ER(2)^*(K(\Zqq,j)) = 0 = ER(2)^*(\Z,j+1) \quad \text{ for } j > 2.
\end{equation*}
In \cite[Theorem 1.3]{Kitch-Lor-Wil-ERn}, we compute
\begin{equation*}
ER(2)^*(K(\Zqq,2)) \quad \text{and} \quad ER(2)^*(\Z,3). 
\end{equation*}
This paper will be devoted to computing 
\begin{equation*}
ER(2)^*(
B \Zqq = K(\Zqq,1) )
\quad
\text{and}
\quad
ER(2)^*(
\C \PP^n). 
\end{equation*}
Together with the second author's computation of 
\[
ER(2)^*(
\C \PP) = ER(2)^*(K(\Z,2)),
\]
this completely solves the problem of computing $ER(2)^*(-)$
for Eilenberg-MacLane spaces.

Numerous other examples are given in \cite{Kitch-Lor-Wil-ERn}.
Somewhere in this mix
is a computation, 
\cite{Kitch-Lor-Wil-BUn},
of 
\begin{equation*}
ER(2)^*(BU(q))
\quad
\text{and} 
\quad
ER(2)^*(\Smash^n \C \PP^\infty).
\end{equation*}

The main tool we use is the
stable cofibration from \cite[Display 1.1]{Nitufib}:
\begin{equation}
\label{stexcp}
\xymatrix{
\Sigma^{17} ER(2)  \ar@{->}[r]^-{x} & ER(2) \ar[r] & E(2)
}
\end{equation}
where $x \in ER(2)^{-17}$, $2x = 0$, and the second map
is the homotopy fixed point inclusion.
This gives rise to a Bockstein spectral sequence (BSS) 
with $E_1 = E(2)^*(X)$
and which collapses after $E_8$ because the self map above 
 has the property that $x^7 = 0$.
Here $x$ generalizes the class $\eta \in \pi_1(BO) = KO^{-1}$
where $\eta^3 = 0 = 2 \eta$.

Because we use a spectral sequence to compute most of our results, 
the answers are often stated in terms of associated graded objects.
In addition to computing the BSS for the spaces of interest, we
describe the Atiyah-Hirzebruch spectral sequence (AHSS) 
for $\C \PP^\infty$ and $\C \PP^n$.
Because there are maps of all of our spaces to $\C \PP^\infty$, the
complete description of $ER(2)^*(\C \PP^\infty)$ in \cite{Vitaly}
maps to our results and solves many of the extensions we leave alone.

The complete description of our results, with all the
various $x^i$-torsion, is somewhat lengthy, and will be presented
in Section \ref{statements}.  However, there are some results that
can be presented in a clean fashion and could be of the most
interest.  We state them here.

The coefficient ring $ER(2)^*$ has two special elements, $\vhat_1$
and $\vhat_2$, that map to $v_1 v_2^{-3}$ of degree 16
and $v_2^{-8}$ of degree 48 respectively 
in $E(2)^*$.  
The element $\vhat_2$ is the periodicity element in $ER(2)^*$ and
so $ER(2)$ is periodic of period 48.  The ring $ER(2)^*$ has a
lot of interesting structure (see appendix, Section \ref{appendix},
and \cite[Proposition 2.1]{NituER2}), 
but if we only look at elements
in degrees multiples of 16, it simplifies dramatically to:
\begin{equation*}
ER(2)^{16*} = \Zp [\vhat_1 , \vhat_2^{\pm 1} ].
\end{equation*}

Since $E(2)$ is complex orientable, we know that $E(2)^*(\C \PP^\infty)
= E(2)^*[[u]]$ with the degree of $u$ equal to 2.  If we define
$\uhat = v_2^3 u$ of degree -16, 
then 
\[
\pontr \lra 
\uhat c(\uhat) 
\]
where
$\uhat c(\uhat) \in 
E(2)^{-32}(\C \PP^\infty)$
and $\pontr \in ER(2)^{-32}(\C \PP^\infty)$, 
where $c$ comes from
complex conjugation and $\pontr$ is a modified first 
Pontryagin class.  
Modulo filtrations we use, $\pontr = - \uhat^2$.

The previously mentioned 
non-immersion results for real projective spaces 
obtained from $ER(2)$ 
came about by looking only at $ER(2)^{8*}(\R \PP^n)$.  There,
higher powers of a generating class existed than in $E(2)^*(-)$.
(\cite[Theorems 1.6 ]{NituP} and
(\cite[Theorems 1.1 ]{NituP2} for $n$ even,
\cite[Theorem 1.3 
]{NituP2}
for $n=16k+1$, and \cite[Theorem 3.2 
]{Romie}
for $n=16k+9$.
Something very similar happens here and can be extracted as
a reasonably presentable theorem.  
For any complex orientable cohomology theory we have
a first 
Pontryagin class  and its $k+1$-st power will be zero in $\C \PP^{2k}$
and $\C \PP^{2k+1}$. 
Because $ER(2)$ is not complex orientable, we do not have this
restriction and are often able to see higher powers of the first
Pontryagin class, which continues to exist for this theory.

\begin{thm}
\label{thmone}
With $ER(2)^{16*} = 
 \Zp [\vhat_1 , \vhat_2^{\pm 1} ]$,
then 
$ER(2)^{16*}(\C \PP^{8k+i}) = $
\[
\begin{array}{lcr}
ER(2)^{16*}[\pontr]
/(\pontr^{4k+1})
& \hspace{.5in} &
i=0 \\
ER(2)^{16*}[\pontr]
/(\pontr^{4k+2},2\pontr^{4k+1})
& \hspace{.5in} &
i=1 \\
ER(2)^{16*}[\pontr]
/(\pontr^{4k+3},2\pontr^{4k+2}, \vhat_1 \pontr^{4k+2})
& \hspace{.5in} &
i=2 \\
ER(2)^{16*}[\pontr]
/(\pontr^{4k+4}, 2\pontr^{4k+2}, 2\pontr^{4k+3}, \vhat_1 \pontr^{4k+2}, \vhat_1 \pontr^{4k+3})
& \hspace{.5in} &
i=3 \\
ER(2)^{16*}[\pontr]
/(\pontr^{4k+4},2\pontr^{4k+3}, \vhat_1 \pontr^{4k+3})
& \hspace{.5in} &
i=4 \\
ER(2)^{16*}[\pontr]
/(\pontr^{4k+4},2\pontr^{4k+3})
& \hspace{.5in} &
i=5 \\
ER(2)^{16*}[\pontr]
/(\pontr^{4k+4})
& \hspace{.5in} &
i=6 \\
ER(2)^{16*}[\pontr]
/(\pontr^{4k+4})
& \hspace{.5in} &
i=7 \\
\end{array}
\]
\[
ER(2)^{16*}(\C \PP^\infty) =  ER(2)^{16*}[[\pontr]]
\]
\end{thm}

The BSS gives the $x^i$-torsion generators precisely, but we use a spectral 
sequence to compute the BSS, so the $x^i$-torsion generators we see are
really for the associated graded object for our 
auxiliary spectral sequence.  
The complete answer
for these spaces gets quite complicated, but to give some insight
here, we'll describe how the elements of $ER(2)^{16*}(-)$ are
related to the $x^i$-torsion.  We know that we can have at
maximum, $x^7$-torsion, but, in fact, our typical case has only
$x$, $x^3$, and $x^7$-torsion.  In two cases, out of eight, for $\C \PP^n$, we
get $x^5$-torsion generators.

\begin{remark}
The exotic higher powers of the first Pontryagin class all
go to zero in $E(2)^*(\C \PP^n)$. 
As such, they
are divisible
by $x$ and thus torsion elements
because of the long exact sequence,
(\ref{stexcp}),
that
gives the exact couple.
In particular, for the $i=2$ and $4$ cases, the elements
$\pontr^{4k+2}$ and $\pontr^{4k+3}$ are $x^4$ times an $x^7$-torsion
generator.  For the $i=3$ case, $\pontr^{4k+2}$ is $x^4$ times
an $x^7$-torsion generator and $\pontr^{4k+3}$ is $x^6$ times an
$x^7$-torsion generator.
For $i=1$ and $5$, the torsion classes, $\pontr^{4k+1}$ and $\pontr^{4k+3}$,
are $x^2$ times $x^5$-torsion generators.  On these classes,
$\vhat_1^j$ is non-zero.  They are all $x^2$ times $x^3$-torsion
generators.
\end{remark}

\begin{thm}
\label{thmtwo}
In the associated graded object used to compute  the reduced
$ER^{16*}(\C \PP^\infty)$,
we have the following $x^i$ torsion generators:
\[
\begin{array}{lcl}
x^1  & \hspace{1in} & (2,\vhat_1) ER(2)^{16*}[\pontr]\{\pontr\} \\
x^3  & \hspace{1in} & \Zq [ \vhat_2^{\pm 1}, \pontr ] \{\pontr^2\} \\
x^7  & \hspace{1in} & \Zq [ \vhat_2^{\pm 1} ] \{\pontr\} \\
\end{array}
\]
\end{thm}

We recall
$\uhat = u v_2^3$ and use
$\hat{F}$, the modified formal group law defined in the next section.
The well known result for $BP^*(-)$, \cite{Land:Flat}, implies
\begin{equation}
E(2)^*(B\Zqq)= E(2)^*[[\uhat]]/[2^q]_{\hat{F}}(\uhat).
\end{equation}
However, in the case of $q=1$, we also get, \cite[Theorem 3.2]{NituP}:
\begin{equation}
ER(2)^*(B\Zq)= ER(2)^*[[\uhat]]/[2]_{\hat{F}}(\uhat).
\end{equation}
The map $B\Zqq \ra B\Zq$
takes $\uhat$ to the
$[2^{q-1}](\uhat) $ sequence
in $E(2)$-cohomology
and $\uhat$ to $z$ (definition of $z$) in
$ER(2)$-cohomology.

For a ring $S$, the notation $S\{a,b\}$ stands for the free $S$-module
on generators $a$ and $b$.

\begin{thm}
\label{thmtwoo}

\ 

\begin{enumerate}
\item
There is a filtration on
$ER(2)^{16*}(B \Z/(2^q))  $
such that the associated graded object is:
\[
 ER(2)^{16*}[[\pontr]]/(2^q \pontr) \oplus  
ER(2)^{16*}[[\pontr]]/(2) \{2^{q-1}\uhat\}.
\]
where $z$ is represented by $2^{q-1} \uhat$.
\item
The elements $z$ and $\pontr$ generate $ER(2)^{16*}(B\Zqq)$, which
can be written in terms of $z \pontr^i$ and $\pontr^i$.
\item
The map 
$ER(2)^{8*}(B\Zqq) \rightarrow
E(2)^{8*}(B\Zqq)$
is an injection. 
\item
The 
extension problems for $2z$, $z^2$, and $2^q \pontr$, can
be solved in 
$E(2)^{16*}(B\Zqq)$
using the series for 
$[2^q](\uhat)$,
$[2^{q-1}](\uhat)$ 
and $c(\uhat)$.
\end{enumerate}
\end{thm}

The setup for the following all comes from 
the work of the second author in
\cite{Vitaly}. We have a norm map:
\begin{equation}
\xymatrix{
E(2)^{*}(B \Zqq) \ar[rr]^{N_*} && 
ER(2)^{*}(B \Zqq).
}
\end{equation}

\begin{defn}
\label{Nres}
We take a restricted norm, and let $\text{\em im}(N_*^{res})$
be the image of the composition:
\[
\Zp[\vhat_1,v_2^{\pm 2}][[\uhat c(\uhat) ]]\{\uhat,v_2 \uhat, z \uhat, 
v_2 z \uhat\} 
\lra
\]
\[
\xymatrix{
E(2)^{*}(B \Zqq) \ar[r]^{N_*} &
ER(2)^{*}(B \Zqq) 
}
\]
where $z = [2^{q-1}](\uhat)$.
\end{defn}

Under the reduction 
$ER(2)^{*}(B \Zqq)  \rightarrow
E(2)^{*}(B \Zqq)$, 
\begin{equation}
N_*(y) \lra y + c(y). 
\end{equation}
Continuing from \cite[Lemma 10.1]{Vitaly}, the image of $\uhat$ is a power
series in $\pontr$, $N_*(\uhat) = \xi(\pontr)$.
In the statement of the next theorem we use the elements 
$\alpha_i \in ER(2)^{-12i}$ for $0 \le i < 4$, which are introduced in 
the next section.
To simplify notation, let $\alpha_{\{0,1,2,3\}}z$  or $\alpha_{\{0-3\}} z$
denote $\{\alpha_0 z = 2z, \alpha_1 z, \alpha_2 z, \alpha_3 z \}$ and so forth.
Similar to the second author's result for $ER(n)^*(\C \PP^\infty)$, 
\cite[Theorem 1.1]{Vitaly},
we
have:

\begin{thm}
\label{normseq}
There is a short exact sequence of modules over $ER(2)^*$
\[
0 \lra \text{im}(N_*^{res}) \lra 
ER(2)^{*}(B \Zqq)  \lra
\frac{ER(2)^*[[\pontr,z]]}{(J)} \lra 0
\]
where $(J)$ is the ideal generated by 
power series representing
$\xi(\pontr)$, 
$z^2$, $\alpha_{\{0-3\}} z$, and $2^{q-1}\alpha_{\{0-3\}} \pontr$,
all computable by algorithms described in the proof of Theorem \ref{thmtwoo}.
\end{thm}

Note that the last map is not a ring map.

The paper is organized as follows.  We do some necessary preliminaries
in Section \ref{preliminaries},
state our BSS results in Section \ref{statements}.
Next, in Section \ref{bss1}, we do our BSS computations for $\C \PP^\infty$,
and in Section \ref{bss2}, our BSS computations for
$B \Zqq$.
After that we describe the AHSS for $\C \PP^\infty$ in
Section \ref{AHSS},   followed by our computation of the
BSS for $\C \PP^n$ and the proof of Theorem \ref{thmone} 
in Section \ref{cpn}.
We finish with Theorem \ref{normseq} in Section \ref{secnorm}.
We include an appendix, Section \ref{appendix},
giving a table for $ER(2)^*$ in its $\Z/(48)$-graded form.

\section{Preliminaries}
\label{preliminaries}

There are many ways to describe $ER(2)^*$, but we will stick
mainly with the description given in
\cite[Remark 3.4]{Kitch-Wil-BO}.
See also the appendix, Section \ref{appendix}.

Although not always convenient, we traditionally call $\vhat_1 \in ER(2)^*$,
$\alpha$. It has degree 16 and maps to $v_1 v_2^{-3} \in E(2)^*$.  
We also have elements $\alpha_i$, $0 < i < 4$,
with degree $-12i$.  We often extend this notation to $\alpha_0 = 2$.  
These elements map to $2v_2^{2i} \in E(2)^*$.   
For the last  non-torsion algebra generator,
we have $w$ of degree -8, which
maps to $\vhat_1 v_2^4 = v_1 v_2 \in E(2)^*$.

Torsion is generated by the element $x \in ER(2)^{-17}$.  It has $2x=0$
and $x^7 = 0$.  
Keep in mind that $ER(2)^*$ is 48 periodic with periodicity element
$\vhat_2$ (mapping to $v_2^{-8}$).  
We use, for efficient notation, $x^{3-6} = \{x^3, x^4, x^5, x^6 \}.$

\begin{fact}
\label{ffact}
$ER(2)^*$ is:
\[
 \Zp[\vhat_1,\vhat_2^{\pm 1}] \{ 1, w, \alpha_1, \alpha_2, \alpha_3 \} 
\quad
\text{with}
\quad
2w = \alpha \alpha_2 = \vhat_1 \alpha_2
\]
\[
\Zq[\vhat_1,\vhat_2^{\pm 1}] \{ x^{1-2},  x^{1-2} w \}
\qquad
\Zq[\vhat_2^{\pm 1}] \{ x^{3-6} \}.
\]
\end{fact}

What makes $ER(2)^*(-)$ computable is the result from
\cite{Nitufib}
that tells us that the fibre of the  fixed point inclusion, 
$ER(2) \lra E(2)$, is just
$\Sigma^{17} ER(2)$ and that the map of $\Sigma^{17}ER(2)$ to 
$ER(2)$ is just $x$ with $x^7 = 0$, i.e. we have the
stable cofibration sequence:
\begin{equation}
\label{exactcouple}
\xymatrix{
\Sigma^{17}ER(2) \ar[rr]^x & &  ER(2) \ar[ddl] \\
 & & \\
 & E(2) \ar[uul]^\partial &
}
\end{equation}

From this, we have an 
exact couple and
a convergent 
Bockstein Spectral Sequence that begins with $E(2)^*(X)$  
and where there can only be differentials $d_1$ through $d_7$.

We have used two versions of this spectral sequence in
the past.  In \cite[Theorem 2.1]{Kitch-Wil-BO} we used the truncated
version that
converges to $ER(2)^*(X)$, but in \cite[Theorem 4.2]{NituP} and \cite{NituP2}
we used the  untruncated
version that converges to zero.  Both versions
contain the same information, but the designated writer for
this paper prefers the one converging to zero because it gets
cleaner as each differential is computed.  The drawback, of course,
is that one must go back to the differentials to reconstruct
$ER(2)^*(X)$.

We give a simplified 
summary of the 
Bockstein Spectral Sequence 
(BSS) for computing $ER(2)^*(X)$ from $E(2)^*(X)$.

\begin{thm}[{\cite{NituP}[Theorem 4.2]}] \ 
\begin{enumerate}
\item
The exact couple, (\ref{exactcouple}), gives a spectral sequence, $E_r$,
of $ER(2)^*$ modules,
starting with
\[
E_1 \simeq E(2)^*(X) \quad \text{and ending with} \quad 
E_{8} = 0.
\]
\item
$
d_1(y) = v_2^{-3}(1-c)(y)$
where
$c(v_i)= -v_i
$
and $c$ comes from complex conjugation.
\item
The degree of $d_r$ is $17r +1$.
\item
The targets of the 
$d_r$ represent the $x^r$-torsion generators of $ER(2)^*(X)$. 
\end{enumerate}
\end{thm}

\begin{defn}
Let $K_i$ be the kernel of $x^i$ on $ER(2)^*(X)$ and let
$M_i$ be the image of $K_i$ in $ER(2)^*(X)/(xER(2)^*(X)) \subset E(2)^*(X)$.
We call $M_r/M_{r-1} \simeq$ image $d_r$ the {\bf $x^r$-torsion generators.}
\end{defn}

\begin{remark}
All of our BSSs in this paper have only even degree elements, so we
always have $d_2 = d_4 = d_6 = 0$.
\end{remark}

For our purposes, it is important to know how this works
for the cohomology of a point.  
The differential $d_1$ commutes with $\vhat_1$ and $v_2^2$.
All that matters here is 
$d_1(v_2) = 2 v_2^{-2}$.

The $E_2$ term becomes $\Zq [\vhat_1,v_2^{\pm 2}]$.
We have $d_3$ commutes with
$\vhat_1$ and $v_2^4$,  and $d_3(v_2^2) = \vhat_1 v_2^{-4}$.

This leaves us with only $\Zq[v_2^{\pm 4}]$.
We have $d_7$ commutes with
$v_2^8 = \vhat_2^{-1}$ and $d_7(v_2^4) = \vhat_2 v_2^{-8}= \vhat_2^2 = v_2^{-16}$,
so $E_8 = 0$.

In terms of describing our $ER(2)^*$ using this approach, we see
that the $x$-torsion is generated by $\Zp [ \vhat_1, v_2^{\pm 2} ] \{ 2 \}$,
the $x^3$-torsion by $\Zq[\vhat_1,v_2^{\pm 4}]\{\vhat_1\}$, 
and the $x^7$-torsion
by $\Zq [ \vhat_2^{\pm 1 } ]$.  The previous description of $ER(2)^*$ is easy
to relate to this now.  The $x$-torsion is given by
$\Zp[\vhat_1, \vhat_2^{\pm 1}]$ on the $\alpha_i$, $0 \le i < 4$.   The 
$x^3$-torsion is generated over 
$\Zq[\vhat_1, \vhat_2^{\pm 1}]$ 
on $\vhat_1= \alpha$ and $w$.
Finally, the $x^7$-torsion is given by $\Zq[\vhat_2^{\pm 1}]$.

Combining the $x$, $x^3$, and $x^7$-torsion, we find that
\[
ER(2)^{16*} = \Zp [ \vhat_1, \vhat_2^{\pm 1}].
\]

The theory $E(2)^*(-)$ is a complex orientable theory so 
$E(2)^*(\C \PP^\infty) = E(2)^*[[u]]$ where $u$ is of degree 2.
For reasons that will become apparent later, we want to ``hat''
this like we did $E(2)^*$.  This follows \cite[pages 235--236]{Kitch-Wil-BO} and
\cite[Section 5]{NituP}.
The only adjustment needed here is to define $\uhat = u v_2^3$,
of degree -16.
For the purposes of our BSS, we write $E(2)^*(\C \PP^\infty) =
E(2)^*[[\uhat]]$.  Since $v_2$ is a unit, this is not a problem.

\begin{remark}
The standard formal group law for $E(2)$ is 
\[
F(x,y) = \sum a_{i,j} x^i y^j \qquad a_{i,j} \in E(2)^{-2(i+j-1)}
\]
with the degrees
of $x$ and $y$ equal to two.  The element $F(x,y)$ also has degree two.
For $a \in E(2)^{2i}$, define $\hat{a} = a v_2^{3i}$.
Also, let $\hat{x} = v_2^3 x$, $\hat{y} = v_2^3 y$, and $\uhat = u v_2^3$.  
Special cases are $\vhat_1 = v_1 v_2^{-3}$ and $\vhat_2 = v_2 v_2^{-9} = v_2^{-8}$.
Now, define
\[
\hat{F}(\hat{x},\hat{y}) = \sum \hat{a}_{i,j} \hat{x}^i \hat{y}^j
= \sum a_{i,j} v_2^{-3(i+j-1)} x^i v_2^{3i} y^j v_2^{3j} 
\]
\[
= v_2^3 \sum a_{i,j} x^i y^j
= v_2^3 F(x,y).
\]
\end{remark}

We will need the complex conjugate of $\uhat$, $c(\uhat)$.  It
has the defining property that $\hat{F}(\uhat,c(\uhat)) = 0$.

The formal group law begins with
\[
\hat{F}(\hat{x},\hat{y}) = \hat{x}+\hat{y} + \vhat_1 \hat{x}\hat{y}
\]
so $c(\uhat)$ begins $-\uhat + \vhat_1 \uhat^2$.
When we work mod 2, the formal group law begins
\[
\hat{F}(\hat{x},\hat{y}) = \hat{x}+\hat{y} + \vhat_1 \hat{x}\hat{y} + 
\vhat_1^2 (\hat{x}^2 \hat{y} + \hat{x} \hat{y}^2)+ \vhat_2 \hat{x}^2 \hat{y}^2
\]
and the corresponding computation mod 2 
begins
\[
c(\uhat) =
\uhat + \vhat_1 \uhat^2 + \vhat_1^2 \uhat^3 + \vhat_2 \uhat^4
\]
While we are working with the formal group law, we need another
fact.  Recall that $[2](\hat{x}) = \hat{F}(\hat{x},\hat{x})$ 
and $[2^q](\hat{x})=
[2]([2^{q-1}(\hat{x})])$.  When we set $0 = [2^q] (\hat{x})$, we need to
know that
\[
0 = 2^q \hat{x} + 2^{q-1}\vhat_1 \hat{x}^2 \mod (\hat{x}^3).
\]
This follows from a simple induction on $q$.

We collect the basics we need:

\begin{lemma} 
\label{lemma}
\[
\begin{array}{lcl}
c(\uhat) = -\uhat + \vhat_1 \uhat^2 & \hspace{.5in} & \mod (\uhat^3) \\
c(\uhat) = \uhat + \vhat_1 \uhat^2 + \vhat_1^2 \uhat^3
+\vhat_2 \uhat^4 
& \hspace{.5in}  &
\mod (2, \uhat^5) \\
\text{If we set } 0=[2^q](\uhat), \text{ then }  
& \hspace{.5in}  &
\ \\
\ 0=[2^q](\uhat) = 2^q \uhat + 2^{q-1}\vhat_1 \uhat^2 
& \hspace{.5in}  &
\mod (\uhat^3) \\
\end{array}
\]
There is 
a modified 
Pontryagin class,
$\pontr \in ER(2)^{-32}(\C \PP^\infty)$, which maps to
$\uhat c(\uhat) \in E(2)^{-32}(\C \PP^\infty)$.

Filtering by powers of $\uhat$, 
$ \pontr
= \uhat c(\uhat) =
-\uhat^2 $
in the
associated graded object.
\end{lemma}

\begin{proof}
The only thing left to prove are the statements about $\pontr$.  We refer
the reader to \cite[Section 5]{Vitaly}.  The sign in the last statement 
follows immediately from the previous statements. 
\end{proof}

\section{Statement of results for the BSS}
\label{statements}

We state all the results for the BSS for $ER(2)^*(-)$ 
for $\C \PP^\infty$, $B \Zqq$, and $\C \PP^n$.
In each case there is a filtration on $ER(2)^*(X)$ such that we
can identify the representatives of the $x^i$-torsion generators
of $ER(2)^*(X)$ in the associated graded object.
This comes about because we use a spectral sequence to compute
the differentials in the BSS.
Keep in mind that the element $\pontr$ is represented by $-\uhat^2$
in the associated graded object.
Also keep in mind that we are working with the entire cohomology here,
not just the degrees $16*$ discussed in the introduction.

\begin{thm}
\label{cp}
There is a filtration on the reduced $ER(2)^*(\C \PP^\infty)$ such that
we can identify the representatives of the $x^i$-torsion
generators for $ER(2)^*(\C \PP^\infty)$ in the associated
graded object as follows:

The $x^1$-torsion generators are:
\[
\Zp[\vhat_1,v_2^{\pm 2},\uhat^2] \{2 v_2 \uhat,
 2  \uhat^2\}
\qquad
\Zq[\vhat_1,v_2^{\pm 2},\uhat^2] \{ \vhat_1 \uhat^2\}.
\]

The $x^3$-torsion generators are
$
\Zq[v_2^{\pm 4}, \uhat^2]\{\uhat^4\}.
$

The $x^7$-torsion generators are
$
\Zq[\vhat_2^{\pm 1} ] \{\uhat^2\}.
$
\end{thm}

\begin{remark}
The reader will note that there is an obvious extension
in the $x^1$-torsion, i.e. 2 times the elements on the right
are in the module on the left.  Honest $x^r$-torsion generators
reduce to elements that are the image of $d_r$.  Since we have
filtered our spectral sequence, none of our $x^r$-torsion generators
are honest because we only see these images in the first filtration
they show up.  In the case of the $x^1$-torsion, we need two differentials
in our filtration in order to compute $d_1$.  The way we describe the
result corresponds to the images of those two differentials.  
\end{remark}

\begin{thm}
\label{zqq}
There is a filtration on the reduced $ER(2)^*(B \Zqq)$ such that
we can identify the representatives of the $x^i$-torsion
generators for $ER(2)^*(B \Zqq)$ in the associated
graded object as follows:

The $x^1$-torsion generators are:
\[
\Z/(2^{q-1})[\vhat_1,v_2^{\pm 2},\uhat^2] \{2v_2 \uhat ,
 2 \uhat^2 \}
\quad
\Zq[\vhat_1,v_2^{\pm 2},\uhat^2] \{ 
\vhat_1 \uhat^2,
2^{q-1} \vhat_1 \uhat^3 \}
\]

The $x^3$-torsion generators are:
\[
\Zq[\vhat_1,v_2^{\pm 4}]\{2^{q-1}\vhat_1 \uhat\}
\qquad
\Zq[v_2^{\pm 4},\uhat^2]\{\uhat^4, 2^{q-1}\uhat^5 \}.
\]

The $x^7$-torsion generators are
$
\Zq[\vhat_2^{\pm 1}]
\{ 2^{q-1} \uhat, \uhat^2 , 2^{q-1}  \uhat^3 \}.
$
\end{thm}
Note that the $x^7$-torsion generators are just $z$, $\pontr$, and $z \pontr$.

The case for $ER(2)^*(\C \PP^n)$ is significantly more
complicated to state.  We break it up into a series of theorems.

\begin{thm}
There is a filtration on the reduced $ER(2)^*(\C \PP^{2j})$ such that
we can identify the representatives of the $x$ and $x^3$-torsion
generators for $ER(2)^*(\C \PP^{2j})$ in the associated
graded object as follows:

The $x^1$-torsion generators are:
\[
\Zp[\vhat_1,v_2^{\pm 2},\uhat^2]/(\uhat^{2j}) \{2 v_2 \uhat,
 2  \uhat^2\}
\qquad
\Zq[\vhat_1,v_2^{\pm 2},\uhat^2]/(\uhat^{2j}) \{ \vhat_1 \uhat^2\}.
\]

The $x^3$-torsion generators are
$
\Zq[v_2^{\pm 4}, \uhat^2]/(\uhat^{2j-2})\{\uhat^4\}.
$
\end{thm}

\begin{thm}
There is a filtration on the reduced $ER(2)^*(\C \PP^{2j+1})$ such that
we can identify the representatives of the $x$ and $x^3$-torsion
generators for $ER(2)^*(\C \PP^{2j+1})$ in the associated
graded object as follows:

The $x^1$-torsion generators are:
\[
\Zp[\vhat_1,v_2^{\pm 2},\uhat^2]/(\uhat^{2j+1}) \{2 v_2 \uhat\}
\qquad
\Zp[\vhat_1,v_2^{\pm 2},\uhat^2]/(\uhat^{2j})\{ 2  \uhat^2\}
\]
\[
\Zq[\vhat_1,v_2^{\pm 2},\uhat^2]/(\uhat^{2j}) \{ \vhat_1 \uhat^2\}.
\]

The $x^3$-torsion generators are:
\[
\Zq[v_2^{\pm 4}, \uhat^2]/(\uhat^{2j-2 })\{\uhat^4\}
\qquad
\Zq [ \vhat_1, v_2^{\pm 4}]\{ \vhat_1 v_2^{2j+1} \uhat^{2j+1}\}.
\]
\end{thm}

\begin{thm}
There is a filtration on $ER(2)^*(\C \PP^{8k+i})$ such that
we can identify the representatives of the $x^5$-torsion
and $x^7$-torsion 
generators for 
the reduced
$ER(2)^*(\C \PP^{8k+i})$ in the associated
graded object as follows:

For all $i$ there are $x^7$-torsion 
generators $\Zq[\vhat_2^{\pm 1}]\{\uhat^2\}$.

The $x^5$-torsion generators are:
\begin{enumerate}
\item[]
For $i=1$:\hspace{.5in}
$\Zq[v_2^{\pm 4}]\{v_2^5 \uhat^{8k+1}\}$
\item[]
For $i=5$:\hspace{.5in}
$\Zq[v_2^{\pm 4}]\{v_2 \uhat^{8k+5}\}$
\end{enumerate}

The rest of the $x^7$-torsion generators are:
\begin{enumerate}
\item[]
For $i=0$:\hspace{.5in}
$\Zq[\vhat_2^{\pm 1}]  \{
v_2^{6} \uhat^{8k}\}
$
\item[]
For $i=2$:\hspace{.5in}
$
\Zq[\vhat_2^{\pm 1}]  \{
v_2^{2} \uhat^{8k+2}\}
$
\item[]
For $i=3$:\hspace{.5in}
$
\Zq[\vhat_2^{\pm 1}] \{v_2^2 \uhat^{8k+2},  v_2^{7} \uhat^{8k+3}\}
$
\item[]
For $i=4$:\hspace{.5in}
$
\Zq [\vhat_2^{\pm 1}] \{ v_2^{2} \uhat^{8k+4}\}
$
\item[]
For $i=6$:\hspace{.5in}
$
\Zq[\vhat_2^{\pm 1}]  \{ v_2^{6} \uhat^{8k+6}\}
$
\item[]
For $i=7$:\hspace{.5in}
$
\Zq[\vhat_2^{\pm 1}] \{v_2^6 \uhat^{8k+6},  v_2^{3} \uhat^{8k+7}\}
$
\end{enumerate}
\end{thm}

\section{
$ER(2)^*(\C \PP^\infty)$
}
\label{bss1}

In this section we give a quick and dirty computation of 
$ER(2)^*(\C \PP^\infty)$.
This is done with complete detail and real finesse in \cite{Vitaly}
for all $ER(n)$, but
what we do here is enough for our purposes.

We filter again and use an auxiliary
spectral sequence to compute the BSS for $\C \PP^\infty$.
As a result, even our BSS computation results are given in terms
of an associated graded object.  We will abuse the notation
and continue to call these terms $E_r$.

\begin{thm}
Filtering $E(2)^*(\C \PP^\infty) = E(2)^*[[\uhat]]$ by powers
of $\uhat$ we give $E_r$ as an associated graded object of
the actual $E_r$ of the BSS for the reduced $ER(2)^*(\C \PP^\infty)$.
\begin{enumerate}
\item[]
$E_1 =   
E(2)^*[\uhat] \{\uhat\}  $
\item[]
$E_2 = E_3 =  
\Zq[v_2^{\pm 2},
\uhat^{2}] 
\{\uhat^{2}\} $
\item[]
$E_4 = E_5 = E_6 = E_7 =  
\Zq[v_2^{\pm 4} ] \{\uhat^2\} $
\item[]
$E_8 = 
0. $
\end{enumerate}
\end{thm}

\begin{proof}
We are first going to compute a spectral sequence for computing $d_1$.
This spectral sequence will collapse after the first two differentials,
which we will call $d_{1,1}$ and $d_{1,2}$.
Our spectral sequence for computing $d_1$ comes from
filtering by powers of $\uhat$.  

From the computation of $d_1$ for $ER(2)^*$, we know that
$d_1$ commutes with $v_2^2$ and $\vhat_1$. 
Filtering by the powers of $\uhat$, we have
$c(\uhat) = - \uhat$ (from Lemma \ref{lemma}) 
in the associated graded object, so
\[
d_{1,1}(\uhat) = v_2^{-3}(1-c)\uhat = v_2^{-3}(\uhat - (-1)\uhat)= 2v_2^{-3}\uhat.
\]
\[
d_{1,1}(\uhat^2) = v_2^{-3}(1-c)\uhat^2 = v_2^{-3}(\uhat^2 - (-1)^2\uhat^2)= 0.
\]
\[
d_{1,1}(v_2 \uhat) = v_2^{-3}(1-c)v_2 \uhat 
= v_2^{-3}(v_2 \uhat - (-1)^2 v_2 \uhat)= 0.
\]
\[
d_{1,1}(v_2 \uhat^2) = v_2^{-3}(1-c)v_2 \uhat^2 
= v_2^{-3}(v_2 \uhat^2 - (-1)^3 v_2 \uhat^2)= 2 v_2^{-2} \uhat^2.
\]

We can now read off the first term of the $x$-torsion in
Theorem \ref{cp}.
After taking the homology with respect to $d_{1,1}$, all
we have left is:
\[
\Zq[\vhat_1, v_2^{\pm 2},\uhat^2]\{ v_2 \uhat \}
\quad
\text{and}
\quad
\Zq[\vhat_1,v_2^{\pm 2}, \uhat^2]\{\uhat^{2}\}.
\]

There is one more computation to do to finish off $d_1$.
The second differential, $d_{1,2}$, in the spectral sequence
to compute $d_1$ 
requires the use of not just $\uhat^j$ but $\uhat^{j+1}$.
We are now working mod 2 so
we have $c(\uhat) = \uhat + \vhat_1 \uhat^2$ (from
Lemma \ref{lemma}).  Recall again
that $d_1$ commutes with $\vhat_1$ and $v_2^2$.  
We have our second differential:
\[
d_{1,2}(v_2 \uhat) = v_2^{-3} (1+c)v_2 \uhat = v_2^{-3} 
(v_2 \uhat + c(v_2 \uhat)) 
\qquad
\qquad
\qquad
\]
\[
\qquad
\qquad
\qquad
= 
v_2^{-3}(v_2 \uhat +  v_2 ( \uhat + \vhat_1 \uhat^2))
=
v_2^{-2} ( \vhat_1 \uhat^2).
\]
\[
d_{1,2}( \uhat^2) = v_2^{-3} (1+c) \uhat^2 = v_2^{-3} 
( \uhat^2 + c( \uhat^2)) 
\qquad
\qquad
\qquad
\]
\[
\qquad
\qquad
\qquad
= 
v_2^{-3}( \uhat^2 +   ( \uhat + \vhat_1 \uhat^2)^2)
=
v_2^{-3} ( 2 \uhat^2 + 2 \vhat_1 \uhat^3) = 0.
\]

This gives us our final term of $x$-torsion in Theorem \ref{cp}.
All that remains after taking the homology with respect to $d_{1,2}$ 
is just our stated $E_2$ term.

With the degree of $v_2$ equal to -6 and the degree of $\uhat$
equal to -16, we see that this is all in degrees multiples of 4,
but the differential $d_1$ has degree 18, which equals 2 mod 4,
so there can be no more to the differential $d_1$, so this is 
an associated graded version of $E_2$.

The differential $d_2$ is odd degree, so it is zero and we have
computed $E_3$.
Although we have only computed an associated graded version of $E_3$, 
we know 
that $2x=0$, and all $x$-torsion was detected by $d_1$.  Consequently,
the real $E_3$ is a $\Zq$-vector space.  Any extensions only 
involve $\vhat_1$.

We have $\vhat_1$ acts trivially on our associated graded version of $E_3$,
but it isn't actually zero on $E_3$.  We need
to solve this extension problem.
Because we know that $\pontr \in ER(2)^*(\C \PP^\infty)$ maps to
$\uhat c(\uhat) \in E(2)^*(\C \PP^\infty)$, $\vhat_1 \pontr \ne 0$ and
so must have a representative in our $E_2 = E_3$.

Working mod 2 and mod $\uhat^5$,
recall 
from Lemma \ref{lemma} that
we have
\[
c(\uhat) =
\uhat + \vhat_1 \uhat^2 + \vhat_1^2 \uhat^3 + \vhat_2 \uhat^4.
\]
Now take 
\[
d_1(v_2^3 \uhat) =
v_2^{-3} (v_2^3 \uhat +  v_2^3(
\uhat + \vhat_1 \uhat^2 + \vhat_1^2 \uhat^3 + \vhat_2 \uhat^4))
= 
  \vhat_1 \uhat^2 + \vhat_1^2 \uhat^3 + \vhat_2 \uhat^4.
\]
This has to be zero in $E_2$ (mod 2 and $\uhat^5$).  
Now, in the first term, substitute 
\[
\uhat 
=
c(\uhat) + \vhat_1 \uhat^2 + \vhat_1^2 \uhat^3 
\]
for one of the $\uhat$ to get
\begin{multline*}
d_1(v_2^3 \uhat) =
  \vhat_1 \uhat^2 + 
\vhat_1^2 \uhat^3 
+ \vhat_2 \uhat^4
\\
=
  \vhat_1 \uhat   
(c(\uhat) + \vhat_1 \uhat^2 + \vhat_1^2 \uhat^3 )
+ \vhat_1^2 \uhat^3 + 
\vhat_2 \uhat^4
\\
=
\vhat_1 \uhat c(\uhat) + 
\vhat_2 \uhat^4 +
\vhat_1^3 \uhat^4 .
\end{multline*}
The last term here does not exist in our $E_2$ so must be
represented in a higher filtration.  Our representative
for $\vhat_1 \pontr$ is thus $\vhat_2 \uhat^4$.  Since
$\pontr$ is represented by $\uhat^2$, we have
\begin{equation}
\label{rel}
0 = \vhat_1 \uhat^2 + \vhat_2 \uhat^4.
\end{equation}
We are ready to compute $d_3$ on
$
\Zq[v_2^{\pm 2},
\uhat^{2}]
\{\uhat^{2}\}.
$
We know that $\uhat c(\uhat) = \pontr$ is a permanent cycle from Lemma \ref{lemma}, 
but this
is, mod $\uhat^5$, just
$
\uhat^2 + \vhat_1 \uhat^3 + \vhat_1^2 \uhat^4.
$
The last two terms would have to be represented in higher filtrations
so they are zero 
mod $\uhat^5$.  So, modulo $\uhat^5$, we have that $d_3(\uhat^2) = 0$.

Compute 
\[
d_3(v_2^2 \uhat^2) = d_3(v_2^2) \uhat^2 
=
\vhat_1 v_2^{-4} \uhat^2
=
v_2^{-4} \vhat_2 \uhat^4
=
v_2^{-12} \uhat^4.
\]

This gives the $x^3$-torsion of Theorem \ref{cp}.
Since $d_3$ commutes with $v_2^4$ and $\uhat^2$, the homology gives $E_4$
as stated.

Elements of $E_4$ are spaced out by the degree of $v_2^4$, or, -24.
The differentials $d_4$, $d_5$, and $d_6$ have degrees 69, 86, and 103,
and these are all non-zero mod 24, so these differentials are all
trivial.

We know that $\uhat c(\uhat)$ is a permanent cycle, so our differential
must be on the $v_2^4$ as in the coefficients.
We get
\[
d_7(v_2^4 \uhat^2) =\vhat_2 v_2^{-8}\uhat^2
= v_2^{-16} \uhat^2 = \vhat_2^2 \uhat^2.
\]

This gives our $x^7$-torsion for Theorem \ref{cp}
and $E_8 = 0$.
\end{proof}

This also completes our description of $ER(2)^{16*}(\C \PP^\infty)$ 
in Theorems \ref{thmone} and \ref{thmtwo}.

\section{$ER(2)^*(B\Zqq)$}
\label{bss2}

As mentioned in the introduction,
\[
E(2)^*(B\Zqq)= E(2)^*[[\uhat]]/[2^q]_{\hat{F}}(\uhat)
\] 
and
\[
ER(2)^*(B\Zq)= ER(2)^*[[\uhat]]/[2]_{\hat{F}}(\uhat).
\] 
As we saw above, $\uhat \in E(2)^{-16}(\C \PP^\infty)$ is not a permanent
cycle in the BSS for $\C \PP^\infty$,
but this result for $\R \PP^\infty = B\Zq$ shows that 
$\uhat \in E(2)^*(B\Zq)$ and all its powers must be permanent cycles
in the BSS for $B\Zq$.
The BSS for this is described in \cite[Theorem 8.1]{NituP}
but redone here.

\begin{thm}
Filtering $E(2)^*(B\Zqq ) = E(2)^*[[\uhat]]/([2^q](\uhat))$ by powers
of $\uhat$ we give $E_r$ as an associated graded object of
the actual $E_r$ of the reduced BSS for $ER(2)^*(  B\Zqq)$.
\begin{enumerate}
\item[]
$E_1 =   
\Zqq [ \vhat_1, v_2^{\pm 1}, \uhat ]\{ \uhat \} $
\item[]
$E_2 = E_3 =
\Zq[\vhat_1,v_2^{\pm 2}] \{ 2^{q-1} \uhat \}
\qquad
\Zq[v_2^{\pm 2},\uhat^2] \{ \uhat^2, 2^{q-1}  \uhat^3\} $
\item[]
$E_4 = E_5 = E_6 = E_7 =
\Zq[v_2^{\pm 4} ] \{ 2^{q-1} \uhat, \uhat^2 , 2^{q-1}  \uhat^3 \} $
\item[]
$E_8 = 
0. $
\end{enumerate}
\end{thm}

\begin{proof}
Since $[2^q](\uhat) = 2^q \uhat$ mod $\uhat^2$, if we start our BSS
for 
$ER(2)^*(B\Zqq)$
not with $E(2)^*(B\Zqq)$, but by filtering this by powers
of $\uhat$, we get an associated graded version of $E_1$ as above.

We have a surjective map $E(2)^*(\C \PP^\infty) \ra E(2)^*(B\Zqq)$.  In
both cases we start by filtering by powers of $\uhat$, so this gives
a map of associated graded rings.  
We use this filtration to compute $d_1$, and
our first differential is inherited from $\C \PP^\infty$ giving the
first part of the
$x$-torsion of Theorem \ref{zqq}.
Taking the homology with respect to this $d_{1,1}$ differential gives a very
different answer from that for $\C \PP^\infty$.  
We get:
\[
\Zq[\vhat_1,v_2^{\pm 2},\uhat^2] 
\{v_2 \uhat, 2^{q-1}\uhat, \uhat^2, 2^{q-1} v_2 \uhat^2 \}
\]
We now need to compute the second differential in our spectral sequence
for $d_1$, i.e. we need to take into consideration $\uhat^j$ 
and $\uhat^{j+1}$.
We need the solution to the
extension problem on our generators 
given by the $2^q$-series modulo $\uhat^3$ (from Lemma \ref{lemma}):
\[
2(2^{q-1} \uhat) = 2^{q-1} \vhat_1 \uhat^2
\quad
 \text{and} 
\quad
2(2^{q-1} \uhat^2) = 2^{q-1} \vhat_1 \uhat^3.
\]
We compute our $d_{1,2}$ on the generators with the continued
understanding that the differential commutes with $\vhat_1$, $v_2^2$,
and $\uhat^2$.
\[
d_{1,2}(v_2 \uhat) = v_2^{-3}(1+c)v_2 \uhat 
= v_2^{-3}(v_2 \uhat +  v_2 (\uhat + \vhat_1 \uhat^2))= 
v_2^{-2}\vhat_1 \uhat^2.
\]
\[
d_{1,2}(2^{q-1} \uhat) = v_2^{-3}(1+c)2^{q-1} \uhat = 2^{q-1} v_2^{-3}(\uhat + (\uhat+\vhat_1 \uhat^2))
\]
\[
= v_2^{-3}(2^q \uhat + 2^{q-1} \vhat_1 \uhat^2)
= v_2^{-3}(2^{q-1} \vhat_1 \uhat^2 + 2^{q-1} \vhat_1 \uhat^2)
=0.
\]
\[
d_{1,2}(\uhat^2) = 0.
\]
\[
d_{1,2}(2^{q-1} v_2 \uhat^2) = 2^{q-1} \uhat^2 d_1(v_2) = 2^q v_2^{-2} \uhat^2 =
2^{q-1} \vhat_1 v_2^{-2} \uhat^3.
\]

This gives us the last of the $x$-torsion elements in Theorem \ref{zqq}.
The homology after this $d_{1,2}$ is the $E_2$ stated in the theorem.
The degrees of all the elements we have left are divisible by 4
and our $d_1$ is of degree 2 mod 4, so we have finished with
our $d_1$.

The map $B\Zqq \ra B\Zq$ induces the map
$E(2)^*(B\Zq) \ra E(2)^*(B\Zqq)$ taking $\uhat$ to
$[2^{q-1}](\uhat) \in E(2)^*(B\Zqq)$.  
Consequently, this must
be a permanent cycle.  In our associated graded $E_3$, this means that 
$2^{q-1}\uhat$ has no differential.  
We have $\uhat^{2k}$ has no differential and so we get
that
$2^{q-1}\uhat^{2k+1}$ has no differential.  
Consequently, the differential $d_3$ in our filtration is
given by:
\[
d_3(2^{q-1} v_2^2 \uhat^{2k+1}) = 2^{q-1} \vhat_1 v_2^{-4}
\uhat^{2k+1}
\quad
\text{and}
\quad
d_3( v_2^2 \uhat^{2k}) =  \vhat_1 v_2^{-4}
\uhat^{2k}.
\]
In the first case, if $k=0$, the $\vhat_1$ is there.  
If $k > 0$ or we look at the second case, we use the
relation inherited from $\C \PP^\infty$ (from \ref{rel}):
$
\vhat_1 \uhat^{2} = \vhat_2 \uhat^{4}.
$
This gives
\[
2^{q-1} \vhat_1 v_2^{-4}
\uhat^{2k+3} =
2^{q-1} v_2^{-4} \vhat_2 \uhat^{2k+5}  \quad \text{and} \quad
\vhat_1 v_2^{-4} \uhat^{2k} = v_2^{-4}\vhat_2 \uhat^{2k+2}.
\]

We can now read off the $x^3$-torsion in the associated graded object 
from this for Theorem \ref{zqq}.
After our $d_3$, we are left with $E_4$ as stated in the theorem.
All elements are in degrees divisible by 8, and the differentials
$d_4$, $d_5$, and $d_6$ have degrees 5, 6, and 7, mod 8, so are
all zero.

Our 3 generators are known to be cycles in this filtration,
so the differential $d_7$ is determined by what happens on the
coefficients,
\[
d_7(v_2^4)= \vhat_2 v_2^{-8} = v_2^{-16} = \vhat_2^2.
\]
Our $x^7$-torsion is as in Theorem \ref{zqq} and $E_8 =0$.
\end{proof}

Observe that if $q=1$, we have $x^7$-torsion generators
$\uhat^{1-3}$ and as $q$ goes off to infinity, we are just left
with our $\uhat^2$ from $\C \PP^\infty$.


\begin{proof}[Proof of Theorem \ref{thmtwoo}]
We have already proven (1) and (2).  
Recall that $ER(2)^{8*}$ is just $ER(2)^*$ without $x$ or $\alpha_1$ and
$\alpha_3$, as in Fact \ref{ffact}.  These coefficients inject into
the coefficients $E(2)^{8*}$.
To show that 
$ER(2)^{8*}(B\Zqq) \rightarrow
E(2)^{8*}(B\Zqq)$
injects, all we need to do is read off the answer from Theorem \ref{zqq}.
All elements in the kernel must have an $x$.  The $x^1$-torsion never
has an $x$ non-zero on it.  All of the generators of $x^3$ and $x^7$-torsion
are in degree zero mod (8) and the degree of $x$ mod (8) is -1.  
The injection follows.

By the injection we know that $z^2$ must 
lie in the description given by Theorem \ref{thmtwoo}, (2).  
In principle,
we know all about
$E(2)^*(B\Zqq)$.
We know that $z$ goes to $[2^{q-1}](\uhat)$ and $\pontr$ 
goes to $\uhat c(\uhat)$.  We also know that $0 = [2^q](\uhat)$.
The formal group law gives the series for $c(\uhat)$ by
$0 = \hat{F}(\uhat, c(\uhat))$.  

We need an algorithm that allows us to compute $2z$, $z^2$, and
$2^q \pontr$ in terms of $z \pontr^i$ and $\pontr^i$.  
We need to be more specific.  
We want to write our elements in terms of a series using
sums of elements $a z \pontr^i$ and $b \pontr^i$ with
$a,b \in \Zp[\vhat_1,\vhat_2]$ and where 2 does not
divide $a$ and $2^q$ does not divide $b$.
From the injection, we know this is possible and our algorithm works
for all three cases.
We use
the same names, $z$ and $\pontr$, for the images in $E(2)^*(B \Zqq)$.

Let
\[
z = [2^{q-1}](\uhat) = f(\uhat) = \sum_{i \ge 0} f_i \uhat^{i+1}
\quad
\text{ with } \quad f_0 = 2^{q-1}
\]
\[
 [2^{q}](\uhat) = g(\uhat) = \sum_{i \ge 0} g_i \uhat^{i+1}
\quad
\text{ with } \quad g_0 = 2^{q}
\]
\[
c(\uhat) =  h(\uhat) = \sum_{i \ge 0} h_i \uhat^{i+1}
\quad
\text{ with } \quad h_0 = -1
\]
We have
\[
f_i, g_i, h_i \in \Zp[\vhat_1,\vhat_2] \subset E(2)^*
\]
and
\[
2^{q-1} \uhat = z - \sum_{i>0} f_i \uhat^{i+1}
\]
\[
2^{q} \uhat =  - \sum_{i>0} g_i \uhat^{i+1}
\]
\[
- \uhat =c(\uhat)  - \sum_{i>0} h_i \uhat^{i+1}
\]
\[
 \uhat^2 = -\uhat (-\uhat) = -\uhat (c(\uhat)  - \sum_{i>0} h_i \uhat^{i+1})
= - \pontr + \sum_{i > 0} h_i \uhat^{i+2}
\]

The important facts here are that, mod higher powers of $\uhat$, 
$z = 2^{q-1} \uhat$, $0 = 2^{q} \uhat$, and $\pontr = - \uhat^2$.

We can now present our algorithm  for 
computing $2z$, $z^2$, and $2^q \pontr$
using induction.  To start, the
only coefficient of $\uhat$ we can have from our elements of 
interest is $2^{q-1} \uhat$ from $z$ and the injection, 
which we can replace with the formula
above to get $z$ modulo higher powers of $\uhat$.

Let us assume that we have succeeded for powers of $\uhat$ below
$\uhat^j$.  If we have a term here, $a \uhat^j$ with $2^q$ dividing
$a$, then we rewrite as
\[
a \uhat^j = (a/2^q)\uhat^{j-1}( - \sum_{i > 0} g_i \uhat^{i+1})
 = 0 \mod \uhat^{j+1}.
\]
If $2^q$ does not divide $a$, we have two cases.
If $j$ is even, say $j=2k$, then 
\[
a \uhat^{2k} = a ( -\pontr + \sum_{i> 0} h_i \uhat^{i+2})^k 
= a( - \pontr)^k \mod \uhat^{j+1}.
\]
If $j$ is odd, say $j=2k+1$, we must use injectivity to see
that $2^{q-1}$ divides $a$, but not $2^q$ (as we have already
dealt with that). Now we can set
\begin{multline*}
a \uhat^{2k+1} = (a/2^{q-1}) (z - \sum_{i>0} f_i \uhat^{i+1})
( -\pontr + \sum_{i> 0} h_i \uhat^{i+2})^k
\\
= (a/2^{q-1}) z (-\pontr)^k \mod \uhat^{j+1}.
\end{multline*}
This concludes the inductive step.
\end{proof}

\begin{remark}
\label{alpharelations}
We can do more than this.
We can
clearly push the elements $2^{q-1}\alpha_i \pontr$ and $\alpha_i z$
into $E(2)^*(B \Zqq)$ and do the same thing as above to get relations.
For example, $\alpha_3 z $ reduces to a series  
with lead term $2 v_2^6 2^{q-1} \uhat$ and so we
have $2^q \uhat$, which lives in higher filtrations.  
We have not given nice names to elements that live in degree
$4 \mod (8)$, so they are not so easy to describe.
From Theorem \ref{zqq}, we can read off that all $x^i$-torsion
generators
in degrees $4*$ inject by the reduction to $E(2)^*(B \Zqq)$.
However, in the case of elements in degree $4*$, there are 3
elements (setting $\vhat_2 = 1$) in the kernel, namely $\{x^4 z, 
x^4 \pontr, x^4 z \pontr \}$.
At first glance it appears that we cannot solve for these relations
completely in $E(2)^*(B \Zqq)$ because of the kernel.  
However, we know we have injection
for degrees $8*$.  The only new relations not in degree $8*$ 
and not divisible by $x$ are
for $2^{q-1} \alpha_{\{1,3\}}\pontr$ and $ \alpha_{\{1,3\}} z$.
The degrees of these four elements, mod (48), are 4, 28, 20, and 44.
The degrees of the 3 elements above that are divisible by $x$
are  12, 44, and 28, so we could possibly have a problem here.
For example, write the above $\alpha_3 z =  y + a x^4 \pontr$
where no term in the series $y$ is divisible by $x$.  
We would like to show
that $a$ must be zero.  Since $\alpha_3$ is $x^1$-torsion, and
$x^4 \pontr$ is $x^3$-torsion, if $a \ne 0$, then $y$ must
be an $x^3$-torsion generator.  However, a quick look at Theorem
\ref{zqq} shows that there are no $x^3$-torsion generators in
degree equal to 4 mod (8).
Only one other element could have a similar problem but it is solved
in the same way.  The bottom line is that all the relations
that involve the $\alpha_i$ can be solved in $E(2)^*(B \Zqq)$
without any involvement of elements divisible by $x$.
\end{remark}

\section{The Atiyah-Hirzebruch spectral sequence for $ER(2)^*(\C \PP^\infty)$}
\label{AHSS}

There are only 3 differentials in the AHSS for $ER(2)^*(\C \PP^\infty)$, 
$d_2$, $d_4$, and
$d_6$.
To simplify our computations here,
we grade over $\Z/(48)$ by setting $\vhat_2 = v_2^{-8} = 1$,
without any loss of information.

Our goal is to use our results from the BSS to compute the AHSS and
then to identify all our terms from the BSS in the AHSS.  This
immediately gives the AHSS for $ER(2)^*(\C \PP^n)$, and then we can
use it to go back and compute the BSS for $\C \PP^n$.
This is a novel circle of arguments and not necessarily the most
efficient approach to $ER(2)^*(\C \PP^n)$, but the AHSS result is of
some interest in its own right.
In particular, it is from the AHSS that we get the results on the 
extra powers  of $\pontr$
for Theorem \ref{thmone}.

In general, when we talk about the AHSS, we will use $\alpha$ for
$\vhat_1$ and $u$ instead of $\uhat = v_2^3 u$, but stick with
$\vhat_1$ and $\uhat$ with the BSS.

The degree of $u$ is 2, and it is the natural choice
for the AHSS,  since 
\[
H^*(\C \PP^\infty;\Zp)  = \Zp [ u ] \text{ and }
H^*(\C \PP^\infty;\Zq) = \Zq [ u ]. 
\]

We describe the AHSS in a sequence of theorems, keeping in mind
that we have set $\vhat_2 = 1$.
To help keep track of degrees, it might be helpful to the reader
to remember the table for $ER(2)^*$ in the appendix, Section
\ref{appendix}.

\begin{thm}
The $E_2$ term of the (reduced)
AHSS for 
$ER(2)^*(\C \PP^\infty)$ is:
\[
\Zp [ \alpha,  u ] \{u,wu, \alpha_1u, \alpha_2u, \alpha_3u \}
\quad
\text{with}
\quad
2w u = \alpha \alpha_2 u.
\]
\[
\Zq [ \alpha,  u ] \{x^{1-2}u, x^{1-2}wu  \}
\qquad
\Zq [  u ] \{x^{3-6} u  \}
\]
The differential, $d_2$, is determined by the multiplicative
structure and
$
d_2(u)= x\alpha u^2.
$

$E_3 = E_4$ is
\[
\Zp [ \alpha,u^2 ] \{ \alpha_i u  \}
\qquad
\Z/(2) [ \alpha,  u^2 ] \{ x^2 \alpha u, x^2 w u  \}
\qquad
0 \le i < 4
\]
\[
\Zp[\alpha,u^2]  \{   u^2 , w u^2, \alpha_1 u^2, \alpha_2 u^2, \alpha_3 u^2 \}  
\quad
\text{with}
\quad
2w u^2 = \alpha \alpha_2 u^2
\]
\[
\Z/(2) [  u^2]
\{  x^{1-6} u^2, x^{1-2}w u^2, x^{2-6} u  \}.
\]
\end{thm}

The reader is spared the complete description of $E_5= E_6$.
We can do this because we know
$ER(2)^*(\C \PP^\infty)$ already.
We just identify 
$ER(2)^*(\C \PP^\infty)$ 
in the AHSS and we can ignore those elements because they cannot
be either source or target for differentials.
We will not identify all elements just yet though.

\begin{thm} 
\label{xtorid}
In $E_4$ of
the AHSS for $ER(2)^*(\C \PP^\infty)$
we identify 
the $x^1$-torsion generators from the BSS
on the left with the AHSS elements on the right.
\[
\Zp[\vhat_1,v_2^{\pm 2},\uhat^2] \{2 v_2 \uhat, 2  \uhat^2\}
=
\Zp [ \alpha,u^2 ] \{ \alpha_i u, \alpha_i u^2  \}
\qquad
0 \le i < 4
\]
\[
\Zq[\vhat_1,v_2^{\pm 4},\uhat^4] \{ \vhat_1 v_2^2 \uhat^2,
\vhat_1 \uhat^4 \}
=
\Zq[\alpha,u^2]\{\alpha u^2, w u^2 \}
\]
\[
\Zq[\vhat_1,v_2^{\pm 4},\uhat^4] \{ \vhat_1  \uhat^2,
\vhat_1 v_2^2 \uhat^4 \}
=
\Zq[\alpha,u^2]\{x^2 \alpha u, x^2 w u \}
\]
\end{thm}

\begin{thm}
After removing the $x^1$-torsion from $E_4$ for the AHSS for
$ER(2)^*(\C \PP^\infty)$, what remains is:
\[
\Zq[u^2]  \{  x^{0-6} u^2  \}
\qquad
\Z/(2) [  u^2]
\{  x^{1-2}w u^2, x^{2-6} u  \}.
\]
The differential, $d_4$, is determined by the multiplicative
structure and 
$d_4(u^2) = x^3 u^4$.

$E_5 = E_6$ is, after removing all of the $x^1$-torsion:
\[
\Z/(2) [  u^2]
\{  x^{1-2}w u^2   \}
\qquad
\Z/(2) 
\{   x^{2-6} u  \}
\qquad
\Zq[u^4]  \{  x^{4-6} u^2  \}
\]
\[
\Zq[u^4]  \{  x^{0-2} u^4  \}
\qquad
\Z/(2) [  u^4 ] \{   x^{4-6} u^3  \}
\qquad
\Z/(2) [  u^4 ] \{   x^{2-4} u^5  \}.
\]
\end{thm}

We can now describe, and eliminate, the $x^3$-torsion elements
from $E_6$ before we compute $d_6$.

\begin{thm} 
\label{x3torid}
In 
$E_6$ of
the AHSS for $ER(2)^*(\C \PP^\infty)$
we identify the elements 
involved with $x^3$-torsion from the BSS
on the left with the AHSS elements on the right.
\[
\Zq[\uhat^8]\{x^{0-2} \uhat^8, x^{0-2} v_2^4 \uhat^4 \} 
= \Zq[u^4]\{x^{0-2} u^4 \}
\]
\[
\Zq[\uhat^8]\{x^{0-2} \uhat^{10}, x^{0-2} v_2^4 \uhat^6 \} 
= \Zq[u^4]\{x^{2-4} u^5 \}
\]
\[
\Zq[\uhat^8]\{x^{0-2} \uhat^{4}, x^{0-2} v_2^4 \uhat^8 \} 
= \Zq[u^4]\{x^{4-6} u^2 \}
\]
\[
\Zq[\uhat^8]\{ \uhat^{6},  v_2^4 \uhat^{10} \} 
= \Zq[u^4]\{x^{6} u^3 \}
\]
\[
\Zq[\uhat^8]\{x^{1-2} \uhat^{6}, x^{1-2} v_2^4 \uhat^{10} \} 
= \Zq[u^4]\{x^{1-2}w u^4 \}
\]
\end{thm}

\begin{thm}
After removing the $x$ and  $x^3$-torsion from $E_6$ for the AHSS for
$ER(2)^*(\C \PP^\infty)$, what remains is:
\[
\Zq[u^4]\{x^{1-2} w u^2 \}
\qquad
\Zq \{x^{2-6} u\}
\qquad
\Zq[u^4]\{x^{4-5} u^3 \}
\]

The differential, $d_6$, is determined by the multiplicative
structure and 
$d_6(x^4 u^3) = x w u^6.$

$E_7 = E_\infty$ is, after removing all of the $x^1$ and $x^3$ torsion:
\[
\Zq \{x^{2-6} u\}
\qquad
\Zq \{x^{1-2}w u^2\}
\]
\end{thm}

\begin{thm} 
In 
$E_\infty$ of
the AHSS for $ER(2)^*(\C \PP^\infty)$
we identify the elements 
involved with $x^7$-torsion from the BSS
on the left with the AHSS elements on the right.
\[
\Zq \{x^{0-4} \uhat^2 \} = 
\Zq \{x^{2-6} u\}
\qquad
\Zq \{x^{5-6} \uhat^2 \} = 
\Zq \{x^{1-2}w u^2\}
\]
\end{thm}

\begin{thm}
\label{powers2}
Let $\pontr $ be the element constructed in  
$ER(2)^*(\C \PP^\infty)$ 
that maps to $\uhat c(\uhat)$, which in the associated graded
object is $\pontr = - \uhat^2$.  Then, in the AHSS, the elements
are represented as follows:
\[
\pontr = x^2 u
\qquad
\pontr^2 = x^4 u^2
\qquad
\pontr^3 = x^6 u^3
\qquad
\pontr^4 =  u^8
\]
\[
\pontr^{4k+1} = x^2 u^{8k+1}
\qquad
\pontr^{4k+2} = x^4 u^{8k+2}
\]
\[
\pontr^{4k+3} = x^6 u^{8k+3}
\qquad
\pontr^{4k+4} =  u^{8k+8}
\]
\end{thm}

The rest of this section consists of the proofs for all of
the theorems stated in this section.
Before we embark on that trip, we offer a small visual guide
at the request of the referee.  Keep in mind that we have
made our theory $ER(2)^*(-)$ 48-periodic, and we only offer
degrees zero through minus 10 for the coefficients here.  We
also truncate as if this was $\C \PP^4$.  Also, many terms have
arbitrarily high powers of $\alpha $ on them.  If a term has an
$x$ in it, is represents a $\Zq$.  If not, it represents a
$\Zp$.  The differential, $d_6$, is not in the range of our
picture.  Neither is $\pontr$ or $\pontr^2$, but $\pontr^3$
is represented here as $x^6 u^3$, in filtration 6 as opposed
to the expected filtration 12.  Note also that the $d_4$ on
$x^3 u^3$ goes off our scale.  That means that if we were really
computing $\C \PP^4$, the element $x^3 u^3$ would have to be a
permanent cycle.
As it stands, all of the elements in the part of the AHSS
pictured below left after $d_2$ and $d_4$ are permanent cycles.

\[
\xymatrix@R5pt{
\    & 2 &  4  & 6  & 8 \\
0      & u \ar[dr]^-{d_2}  & u^2 \ar[rrddd]^(.2){d_4} & u^3  \ar[dr]^-{d_2} & u^4 \\
-1    & x\alpha u \ar[dr]^-{d_2} & x \alpha u^2 & x\alpha u^3 \ar[dr]^-{d_2} & x \alpha u^4 \\
-2    & x^{2}\alpha^{2} u & x^{2} \alpha^{2} u^2 & x^{2}\alpha^{2} u^3 & x^{2} \alpha^{2} u^4 \\
-3    & x^{3} u & x^{3}  u^2  \ar[rrddd]^(.2){d_4}& x^{3} u^3 \ar[rdd]^{d_4} & x^{3}  u^4 \\
-4      & \alpha^2 \alpha_3 u  & \alpha^2 \alpha_3 u^2 & \alpha^2 \alpha_3 u^3  & \alpha^2 \alpha_3 u^4 \\
-5 &&&& \\
-6    & x^{6} u & x^{6}  u^2 & x^{6} u^3 & x^{6}  u^4 \\
-7 &&&& \\
-8      & w u  \ar[dr]^-{d_2} & w u^2 & w u^3  \ar[dr]^-{d_2} & w u^4 \\
-9    & x\alpha w u \ar[dr]^-{d_2} & x \alpha w u^2 & x\alpha w u^3 \ar[dr]^-{d_2} & x \alpha w u^4 \\
-10    & x^{2}\alpha^{2}  w u & x^{2} \alpha^{2} w u^2 & x^{2}\alpha^{2} w u^3 & x^{2} \alpha^{2} w u^4 \\
}
\]

\begin{proof}[Proofs of all the theorems]
We take a short side trip to think about the BSS for $\C \PP^2$.
It
starts with $E(2)^*$ free on $\uhat$ and $\uhat^2$.  The computation
of $d_1$ is identical to that for $\C \PP^\infty$
by naturality
and we have $d_1(v_2 \uhat)$ is non-zero,
but $v_2 \uhat= v_2 (v_2^3 u) = v_2^4 u$ and $d_1$ commutes with $v_2^2$, so
$d_1$ is non-zero on $u$, so $u$ does not exist in $E_\infty$ for the AHSS.

In the AHSS for $\C \PP^2$, the only way $u$ can go away is if
$d_2(u) = x \alpha u^2$, since this is the only
element in the degree of the image.  
Technically, we could have
$d_2(u) = x \alpha^{3k+1} u^2$, but this would immediately conflict
with the answer we get from the BSS.
So, 
$d_2(u) = x \alpha u^2$
is what must happen, and by
naturality, this happens in the AHSS for $\C \PP^\infty$ as well.

Because $2x = 0$, $d_2(u^2) = 0$.  All we need to do to compute
$E_3$ is find the kernel of $d_2$ on $u^{2k+1}$ and the cokernel
on $u^{2k}$.
We know that $x \alpha_i = 0$, where we include $\alpha_0 = 2$.
We also know that $x^3 \alpha = 0$.
The computation of $E_3$ is a straightforward and as stated.

For degree reasons, there is no $d_3$ (or any $d_r$, $r$ odd), 
so we have $E_4 = E_3$.

Before we compute our $d_4$, we want to identify
all of the $x$-torsion generators and eliminate them from
consideration, that is, make the identifications in
Theorem \ref{xtorid}.  We begin with the first term of
the $x$-torsion from Theorem \ref{cp}.  Making the substitution
$\uhat = v_2^3 u$, and working over $\Zp[v_2^{\pm 2}]$, this
turns into 
\[
\Zp [ \vhat_1, v_2^{\pm 2}, u^2 ] \{ 2u, 2u^2 \}.
\]
Making the identification, $2 v_2^{2i} = \alpha_i$, we find we
have the first line of Theorem \ref{xtorid}.

Now break up the second term of the $x$-torsion in Theorem \ref{cp}
into the parts remaining on the left of Theorem \ref{xtorid}.
With the usual substitutions $\uhat = v_2^3 u$ and 
$w = \vhat_1 v_2^4$,
the second line of Theorem \ref{xtorid} becomes obvious, although
it must be kept in mind that the generators don't correspond, 
but
it
is the whole module that is the same.  

The final line of Theorem \ref{xtorid} is 
somewhat more of a challenge.  Here we recall, looking at
the image in $E(2)^*$,
that 
$v_2^2  $ and $v_2^6 $ are 
$\frac{1}{2} \alpha_1 $
and
$\frac{1}{2} \alpha_3 $
respectively.  
To find a representative for 
$\alpha \frac{1}{2} \alpha_1 u^2 $
and
$\alpha \frac{1}{2} \alpha_3  u^2$,
the only elements available are $x^2 w u$ and $x^2 \alpha u $
respectively. 
This, together with our usual substitution of $\uhat = v_2^3 u$,
is enough to give us our last line in the identification.

After taking all of the $x$-torsion from the BSS out of
$E_4$ of the AHSS, we get the stated remaining terms in $E_4$.

Recall that $\uhat = v_2^3 u$ so $u^4 = \uhat^4 v_2^{-12}$ 
and that this is in the image
of $d_3$ in the BSS, so it is both a permanent cycle and $x^3$-torsion.
So, $x^3 u^4$ must be the target of a differential.  The differential
$d_2$ did not hit it.  
The only possible $d_4$ is $d_4(u^2) = x^3 u^4 $.
In principle, a $d_6$ could hit it, but this would require 
$d_6(z_2 u) = x^3 u^4$  for some element
$z_2$ 
with the 
degree of $z_2$ 
equal to 2, and
there is no such element.
We conclude that we must have 
$d_4(u^2) = x^3 u^4 $.

Other than $u^2$, we have one other 
generator (over $ER(2)^*$) in what remains of $E_4$
that $d_4$ could be non-zero on.  
It is $x^2 u$.

The element $x^2 u$ is the most interesting.  We know that
$\uhat c(\uhat)$ exists.  In the AHSS this would be
$- u^2 v_2^6 = - \frac{1}{2} \alpha_3 u^2$.  This doesn't
exist in the filtration associated with $u^2$ in the AHSS,
but $\alpha_3 u^2$ does.  In order to divide it by 2, we
have to go to the previous filtration, where we find $x^2 u$,
which must represent $\uhat c(\uhat)$, and so cannot have
a differential on it.

Now we are free to compute $d_4$ from 
$d_4(u^2) = x^3 u^4 $.
We get the obvious 
\[
d_4(x^{0-3} u^{4k+2}) = x^{3-6} u^{4k+4}.
\]
In addition, we know that multiplication by $x^2 u$ commutes
with $d_4$, so we get another family:
\[
d_4(x^{2-3} u^{4k+3}) = x^{5-6} u^{4k+5}.
\]

There is so little left
of $E_4$ without the $x$-torsion that $d_4$ is easy to compute
and gives the stated result for $E_5 = E_6$ without the
$x$-torsion.

Before we continue on to $d_6$, we want to identify
the $x^3$-torsion from the BSS computation.
Rewrite the $x^3$-torsion from the BSS of Theorem \ref{cp}
into the terms on the left of 
Theorem \ref{x3torid}. 

After that, make the substitution $\uhat = v_2^3 u$, keeping in
mind our ever present $v_2^8 = 1$.  
With this, we have $\uhat^2 = v_2^6 u^2$, $\uhat^4 = v_2^4 u^4$,
$\uhat^6 = v_2^2 u^6$, and $\uhat^8 = u^8$.
The terms on the left
of Theorem \ref{x3torid} reduce to, in order, 
\[
\Zq [u^4]\{x^{0-2} u^4,x^{0-2} v_2^6 u^6, x^{0-2} v_2^4 u^4 , 
x^{0-2} v_2^2 u^6 \}
\]

The one easy case now is the first one, and it gives the first
line of Theorem \ref{x3torid}.

For the second line, we look at
$v_2^6 u^6 = \frac{1}{2}\alpha_3 u^6$.
This must be some element $z_i u^i$ in our $E_6$, 
with $0 < i < 6$, and
our $z_i$ must have
degree $-24 -2i$.  The only possibility is $ x^2 u^5$.
This takes care of the second line.

Next we need to consider
$v_2^4 u^4 = \frac{1}{2}\alpha_2 u^4$.
For this element, we want
some $z_i u^i$, with $0 < i < 4$, and
the degree of $z_i $ to be $-16 -2i$.  
The only possible element here is $x^4 u^2$.  
This completes the third line.

The fourth line is for
$v_2^2 u^6 = \frac{1}{2}\alpha_1 u^6$.
This must be some element $z_i u^i$, with $0 < i < 6$, and
the degree of $z_i $ equal to $-2i$.  The only element that meets
this criteria is $ x^6 u^3$, so this must be it.  Unlike the
others, this one does not have $x^{1-2}$ times it non-zero
in the AHSS.  So, we are missing the fifth line.

The unfortunate consequence of this is that we can not remove all the
$x^3$-torsion from $E_6$ before we compute $d_6$, but
we have to leave the
\[
\Zq[u^4]\{ x^{1-2} w u^4 \}
\]
in $E_6$.  So, after taking out all $x$-torsion and most of the
$x^3$-torsion, all we have left in $E_6$ is
\[
\Zq[u^4]\{x^{1-2} w u^2 \}
\qquad
\Zq \{x^{2-6} u\}
\]
\[
\Zq[u^4]\{x^{4-5} u^3 \}
\qquad
\Zq[u^4]\{ x^{1-2} w u^4 \}
\]
Now comes a tricky part. We have identified our $x^7$-torsion
generator, $\pontr = x^2 u$, already.  The second term here can
be eliminated as $x^{0-4} \pontr$.  We are now missing 3 bits of
information.  We do not know $x^{5-6} (x^2 u)$, $x^{1-2} (x^6u^3)$,
and any differentials that we might have.

Either we have $x^5 (x^2 u) = x w u^2$, or there must be a differential
on $x w u^2$, i.e. $d_{2i}(xw u^2) = z u^{2+i}$ where the degree of $z$
must be
the degree of $xw$ minus $2i-1$ (mod 48).  There is no such $z$
left in our $E_6$.  Consequently, we have $x^5 (x^2 u)= x w u^2$.

Next, either $x (x^6 u^3) = x w u^4$, or there must be a differential
on the last term.  A similar argument to the above shows there is
no such element.  With these extension problems solved, all that
is left of $E_6$ is:
\[
\Zq[u^4]\{x^{1-2} w u^6 \}
\qquad
\Zq[u^4]\{x^{4-5} u^3 \}
\]
The only way they can go away is to have $d_6(x^4 u^3) = x w u^6$.

Theorem \ref{powers2} is now just a matter of inspection.

This completes the proof of all the theorems in this section.
\end{proof}

\begin{remark}
We don't need this, but it is an interesting observation so we
comment on it.  We just showed that in the AHSS, the solutions
to some extension problems are now clear.  We have
\[
2(x^2 u) = \alpha_3 u^2
\qquad
2(x^2 u)^2 = 2 x^4 u^2 =  \alpha_2 u^4
\]
\[
2(x^2 u)^3 = 2 x^6 u^3 =  \alpha_1 u^6
\]
Although the powers of $\pontr$, represented by $x^2 u$, are not
in the expected filtrations, 2 times them are.
\end{remark}

\section{$ER(2)^*(\C \PP^n)$}
\label{cpn}

We can first
note that we have already computed the AHSS for $ER(2)^*(\C \PP^n)$
because the $E_2$ term from the AHSS for $\C \PP^\infty$ surjects.
That means that the only differences are given by the differential
$d_2$ on the $u^n$ filtration, $d_4$ on the $u^{n,n-1}$ filtrations,
and $d_6$ on the $u^{n,n-1,n-2}$ filtrations.
Because there are sometimes no targets, elements survive that
are not in the image from $ER(2)^*(\C \PP^\infty)$.
We will use these facts to compute the BSS for $ER(2)^*(\C \PP^n)$.
As usual for our evenly graded BSSs, we have $d_{2,4,6} = 0$.

To be able to combine calculations we need a notational
convention.  
\[
\text{by} \quad
a^{\{b,c\}} 
d^{\{e,f\}} 
\quad
\text{we mean}
\quad
a^b d^e
\quad
\text{and}
\quad
a^c d^f.
\]
We have stated the description of $ER(2)^*(\C \PP^n)$ already in
Section \ref{statements} and Theorem \ref{thmone}.  Here we
describe the BSS and prove everything.
Our computations are complicated due to having 8 distinct cases.

\begin{thm}
Filtering $E(2)^*(\C \PP^{2j}) = E(2)^*[\uhat]/(\uhat^{2j+1})$ by powers
of $\uhat$ we give $E_r$ as an associated graded object of
the actual $E_r$ of the BSS for 
the reduced
$ER(2)^*(\C \PP^{2j})$.
\begin{enumerate}
\item[]
$E_1 = 
E(2)^*[\uhat]/(\uhat^{2j})
 \{\uhat\} $
\item[]
$E_2 = E_3 =
\Zq[v_2^{\pm 2}, \uhat^{2}]/(\uhat^{2j}) \{\uhat^{2}\}$
\item[]
$E_4 = E_5 = E_6 = E_7 =
\Zq[v_2^{\pm 4} ] \{\uhat^2, v_2^2 \uhat^{2j}\}$
\end{enumerate}
\end{thm}

\begin{thm}
Filtering $E(2)^*(\C \PP^{2j+1}) = E(2)^*[\uhat]/(\uhat^{2j+2})$ by powers
of $\uhat$ we give $E_r$ as an associated graded object of
the actual $E_r$ of the BSS for 
 the reduced
$ER(2)^*(\C \PP^{2j+1})$.
\begin{enumerate}
\item[]
$E_1 = 
E(2)^*[\uhat]/(\uhat^{2j+1})
 \{\uhat\} $
\item[]
$E_2 = E_3 =
\Zq[v_2^{\pm 2}, \uhat^{2}]/(\uhat^{2j})
\{\uhat^{2}\}
\qquad
\Zq[\vhat_1, v_2^{\pm 2}]\{ v_2 \uhat^{2j+1} \}$
\item[]
$E_4 = E_5 = 
\Zq[v_2^{\pm 4}]\{\uhat^2,
v_2^2 \uhat^{2j},
  v_2^{2j+1} \uhat^{2j+1} \} $
\item[]
$E_5 = E_6 = E_7 
\quad 
\text{for } j=4k+1 \text{ and } 4k+3  $
\item[]
$E_6 = E_7 =
\Zq[v_2^{\pm 4}]\{\uhat^2\}
\quad
\text{ for } j=4k \text{ and } 4k+2. $
\end{enumerate}
\end{thm}

The rest of this section is dedicated to the computations proving 
the various theorems about $ER(2)^*(\C \PP^n)$.
Up to $E_5$, the computations pretty much follow from our work
with $\cp$.  After that, technically there are 8 cases.  We present
here a reference guide (at the request of the referee) for $d_5$
and $d_7$ in these 8 cases.  We will ignore the $\uhat^2$ and
$v_2^4 \uhat^2$ terms because the $d_7$ there comes from $\cp$.
We can get all 8 cases from just 3 diagrams.
\[
\xymatrix@R1pt{
&&v_2^2 \uhat^{4k} \ar[r]^-{d_5} & v_2 \uhat^{4k+1}&& \\
 \C \PP^{4k+1} && && & \\
&&v_2^6 \uhat^{4k} \ar[r]^-{d_5} & v_2^5 \uhat^{4k+1}& &\\
}
\]
\[
\xymatrix@R1pt{
&&v_2^2 \uhat^{8k+2}  & v_2^3 \uhat^{8k+3} \ar[dd]^-{d_7}   &&  \\
 \C \PP^{8k+3} && && & \\
&&v_2^6 \uhat^{8k+2} \ar[uu]^-{d_7} & v_2^7 \uhat^{8k+3}& &\\
}
\]
\[
\xymatrix@R1pt{
&&v_2^2 \uhat^{8k+6}\ar[dd]^-{d_7}  & v_2^3 \uhat^{8k+7}    &&  \\
 \C \PP^{8k+7} && && & \\
&&v_2^6 \uhat^{8k+2}  & v_2^7 \uhat^{8k+7}\ar[uu]^-{d_7}& &\\
}
\]
This displays the results for the odd spaces, but we can restrict
$8k+3$ and $8k+7$ to $8k+2$ and $8k+6$ respectively to get the
appropriate $d_7$ for those spaces.  Then, $8k$ and $8k+6$ are the
same and $8k+2$ and $8k+4$ are the same.  This should provide
a guide for when the proofs get confusing.


\begin{proof}[Proofs]

For even spaces,
our BSS $E_1$ is 
$
E(2)^*(\C \PP^{2j}) = E(2)^*[\uhat]/(\uhat^{2j+1}) $.
Since $E(2)^*(\C \PP^\infty)$ surjects to this, we inherit 
our computation of $d_{1,1}$, $d_{1,2}$ and, consequently,
the entirety of $d_1$.
The $x$-torsion and $E_2$ are as stated.

We also inherit
$d_3$ from $\C \PP^\infty$ and can read off the $x^3$-torsion
and $E_4$ directly.
Note that, at this stage, the 
only difference between $\C \PP^{2j}$ and $\C \PP^\infty$
is the lack of a $d_3$ on $v_2^2 \uhat^{2j}$.

For degree reasons, we can only have $d_7$, not $d_5$, so $E_4 = E_7$.
On the first term, $d_7$ is inherited from $\C \PP^\infty$,
but on the last term, there are several cases.

We consider the map:
\[
\C \PP^{8k+2i} \lra S^{16k+4i}.
\]
In $E(2)^*(-)$, the generator for the sphere, $s_{16k+4i}$ maps
to $u^{8k+2i}$.  The generator for the sphere must
be a permanent cycle, so must $u^{8k+2i} \in E(2)^*(\C \PP^{8k+2i})$.
We have $\uhat = v_2^3 u$, so we know that the following is a permanent cycle 
\[
u^{8k+2i} = (v_2^{-3} \uhat)^{8k+2i} 
= v_2^{-24k-6i} \uhat^{8k+2i}
= v_2^{-6i} \uhat^{8k+2i}
= v_2^{2i} \uhat^{8k+2i}.
\]
Unfortunately, this only exists in $E_7$ for $i=\{1,3\}$.
So, we have
\[
d_7(v_2^4 (v_2^{2i} \uhat^{8k+2i})) = 
v_2^{2i} \uhat^{8k+2i}
\qquad
i = \{1,3\}.
\]

This gives the $x^7$-torsion as stated for $8k+2$ and $8k+6$.

We now have to deal with $
8k$ and $8k+4$
where we cannot use
the sphere and naturality.  We'll try to do both at the
same time even though they have different outcomes.
We have two elements in each case, 
$v_2^{\{2,6\}} \uhat^{8k}$  and
$v_2^{\{2,6\}} \uhat^{8k+4}$. 

In each case, one
must be the
source and one the target.  If a target, it has to exist
in the AHSS, and would be represented by one of
$
v_2^{\{2,6\}} \uhat^{8k}  =
v_2^{\{2,6\}} u^{8k}$, and, respectively,
$v_2^{\{2,6\}} \uhat^{8k+4} =
v_2^{\{6,2\}} u^{8k+4}$.
It is important to 
note how in the $8k+4$ case, the powers of $v_2$ got 
switched around in the transition from $\uhat$ to $u$.  

The element we are searching for (the target)
cannot be in the image from
$\C \PP^\infty$ because there is no $x^7$-torsion element there
that could come and hit it.  
These elements
cannot be represented in the filtration occupied by $u^{8k+\{0,4\}}$
because 
$v_2^2 u^{8k+\{0,4\}} =  \frac{1}{2}\alpha_1 u^{8k+\{0,4\}}$ and
$v_2^6 u^{8k+\{0,4\}} =
 \frac{1}{2}\alpha_3 u^{8k+\{0,4\}}$ do not exist there. 
Because they cannot be in the image from $\C \PP^\infty$,
they must be a source for
an AHSS differential in $\C \PP^\infty$ but not in $\C \PP^n$. 
So, the class we search for must be in the filtration of
$u^{8k+\{-1,3\}}$. 
(There is no $d_6$ on the filtration for
$u^{8k+\{-2,2\}}$ and so everything in that filtration is in the
image from $\C \PP^\infty$ and not what we are looking for.)  
The $v_2^2$ case 
would have to be $z_1 u^{8k+\{-1,3\}}$ with
degree of $z_1$ equal to -10 = 38, and the $v_2^6$ case would have to
be
$z_2 u^{8k+\{-1,3\}}$ with
degree of $z_2$ equal to -34=14. The $v_2^2$ candidate would
have to be $z_1 = x^2 w \alpha^{3i+2}$, but we have already shown
that this is a permanent cycle representing an element.  On the
other hand, a good choice for $z_2$ is $x^2 $.  Since we know
that for $\C \PP^\infty$ there is an AHSS $d_4$ on $x^2 u^{8k+\{-1,3\}}$, this
must be our choice.  We could track down the $x^7$-torsion elements
in the AHSS, but it isn't necessary.  We have, in the BSS,
\[
d_7(v_2^2 \uhat^{8k}) =
v_2^6 \uhat^{8k}
\qquad
d_7(v_2^6 \uhat^{8k+4}) =
v_2^2 \uhat^{8k+4}
\]
and our $x^7$ torsion is generated as stated.

For odd $n$,
our BSS $E_1$ is 
$
E(2)^*(\C \PP^{2j+1}) = E(2)^*[\uhat]/(\uhat^{2j+2}).
$
Since $E(2)^*(\C \PP^\infty)$ surjects to this, we inherit 
our computation of $d_{1,1}$, $d_{1,2}$ and, consequently,
the entirety of $d_1$.
We can read off the $x$-torsion and $E_2$ as stated.

We also inherit $d_3$ on the first part of the stated $E_3$
and can read off the $x^3$-torsion it gives as well as
its contribution to $E_4$.
The last part of $E_3$ is different.
In order to determine this, we use the map
\[
\C \PP^{2j+1} \lra S^{4j+2}.
\]
In $E(2)^*(-)$, the generator for the sphere, $s_{4j+2}$ maps
to $u^{2j+1}$.  Since we know the generator for the sphere must
be a permanent cycle, so must $u^{2j+1} \in E(2)^*(\C \PP^{2j+1})$.
We have $\uhat = v_2^3 u$, so the permanent cycle is
\[
u^{2j+1} = (v_2^{-3} \uhat)^{2j+1} 
= v_2^{-6j-3} \uhat^{2j+1}
= v_2^{2j-3} \uhat^{2j+1}.
\]
These elements are all there in $E_3$ so we know $d_3$
by naturality from the sphere, which is just like
the coefficients. It is just:
\[
d_3(v_2^2  (v_2^{2j-3} \uhat^{2j+1}))
= \vhat_1 v_2^{-4} v_2^{2j-3} \uhat^{2j+1}
= \vhat_1  v_2^{2j+1} \uhat^{2j+1}.
\]
As a result, the $x^3$-torsion from this part is as stated and
likewise for the contribution to $E_4$.

We only have $d_5$ and $d_7$
left.  The first term is dealt with by a $d_7$ from $\C \PP^\infty$.

We know that $v_2^{2j-3} \uhat^{2j+1}$ is a permanent cycle.  That means
it must be in the image of either $d_5$ or $d_7$.

There are two ways to go about determining this.  We could
compute $ER(2)^*(-)$ for $\C \PP^{2j+1}/\C \PP^{2j-1}$ and look
at the map of BSSs from this to $\C \PP^{2j+1}$.  This is 
similar to what 
the first and third author
did with $\R \PP^{2n}$ in \cite{NituP}.
Or, 
the way we do it here, is to
look at the AHSS to determine which differential it is.

For degree reasons, the only possible candidates for 
the BSS
$d_5$ are:
\[
d_5(v_2^{\{2,6\}}\uhat^{8k+\{0,4\}}) = v_2^{\{1,5\}} \uhat^{8k+\{1,5\}}
\]
and, since $d_5$ commutes with $v_2^4$ (from the coefficients),
\[
d_5(v_2^{\{6,2\}}\uhat^{8k+\{0,4\}}) = v_2^{\{5,1\}} \uhat^{8k+\{1,5\}}
\]
Recall that $v_2^{2j-3} \uhat^{2j+1} = u^{2j+1}$ is a 
BSS
permanent
cycle, so $v_2^5 \uhat^{8k+1} = u^{8k+1}$ and $v_2 \uhat^{8k+5}
= u^{8k+5}$ are 
both BSS and AHSS
permanent cycles
for $n=8k+1$ and $n=8k+5$, respectively.
In the AHSS,
we know that
$d_4(x^2 u^{8k+\{-1,3\}}) = x^5 u^{8k+\{1,5\}}. $ 
Thus, we must have a BSS $d_5$ because $u^{8k+\{1,5\}}$ is $x^5$-torsion.
We get, in the BSS, for purely dimensional
reasons,
that the candidate $d_5$ above is the correct differential.

Our $x^5$-torsion is as stated as well as the $E_6 = E_7$.
The $E_7$
inherits $d_7$ from $\C \PP^\infty$ so our $x^7$ torsion is
is as stated.

We have seen that the only possibilities for $d_5$ occur, so
for $8k+\{3,7\}$, we are left with only $d_7$.
Recall $E_7$ is:
\[
\Zq[v_2^{\pm 4}]\{\uhat^2,
v_2^2 \uhat^{8k+\{2,6\}},
  v_2^{2j+1} \uhat^{8k+\{3,7\}} \} 
\]
The first term is dealt with by a $d_7$ from $\C \PP^\infty$.
The degrees aren't right for $d_7$ to go from the
second term of $E_7$ to the third, so each term must disappear on its
own from a $d_7$.

The third term is easy because it comes directly from the sphere
so we can read off our $x^7$-torsion generators from it.

The second term is easy as well since we can map to the space
$\C \PP^{8k+\{2,6\}}$, where we know the differential on these terms.
We can read off the new $x^7$-torsion from the middle term.
And, of course,  there is the $x^7$-torsion from
$
\Zq  \{\uhat^2\}.
$

All that remains is to prove Theorem \ref{thmone}.
We can just read off all the elements in degrees $16*$
from our theorems and then use Theorem \ref{powers2}
to identify them with powers of $\pontr$.
\end{proof}

\section{The Norm}
\label{secnorm}

This section is devoted to a proof of Theorem \ref{normseq}.
The analogous result for $ER(n)^*(\C \PP^\infty)$ in \cite{Vitaly}
requires significant theory, mainly because of the arbitrary $n$.
However, in our case, with $n=2$, we have explicitly written down
all of the elements in 
$ER(2)^*(B \Zqq)$ 
in Theorem \ref{zqq}, so
our job here is just a matter of chasing these elements around.

We refer the reader to \cite{Vitaly} for all necessary background
on the norm, $N_*$, and its behavior for $\C \PP^\infty$.  
In particular, because $z$ is an element of $ER(2)^*(B \Zqq)$, the
norm commutes with it.  
Following \cite{Vitaly}, Section 7, 
and our Definition \ref{Nres}, the reduction of 
$\text{im}(N_*^{res})$ 
to $E(2)^*(B \Zqq )$ is generated by
\[
\Zp[\vhat_1,v_2^{\pm 2}][[\uhat c(\uhat)]]\{
\uhat + c(\uhat), v_2 (\uhat - c(\uhat),
z(\uhat + c(\uhat)), v_2 z (\uhat - c(\uhat))
\}.
\]
In \cite[Lemma 10.3]{Vitaly},
it is also shown that 
$\text{im}(N_*^{res})$ 
is $x$-torsion.  
This is enough background to get us started with our proof.

We outline our proof as it is somewhat technical.  From
Theorem \ref{zqq}, we know the $x^1$-torsion and 
$\text{im}(N_*^{res})$ 
lies in this $x^1$-torsion.  We first identify this image.
After that, we take what is left of Theorem \ref{zqq}
and rewrite it in terms of $ER(2)^*$, $z$, and $\pontr$.
From there it is easy to see a surjective map of
$ER(2)^*[[\pontr,z]]/(J)$ 
to $ER(2)^*(B \Zqq)/( \text{im}(N_*^{res})$.  Then, all that
remains is to identify these two.

\begin{proof}[Proof of Theorem \ref{normseq}]
We start by computing the associated graded object for
\[
ER(2)^*(B \Zqq)/( \text{im}(N_*^{res})).
\]
Recall that we
have $ER(2)^*(B \Zqq)$ from Theorem \ref{zqq}:

The $x^1$-torsion generators in (the associated graded
object for) $ER(2)^*(B \Zqq) $ are:
\[
\Z/(2^{q-1})[\vhat_1,v_2^{\pm 2},\uhat^2] \{2v_2 \uhat ,
 2 \uhat^2 \}
\quad
\Zq[\vhat_1,v_2^{\pm 2},\uhat^2] \{ 
\vhat_1 \uhat^2,
2^{q-1} \vhat_1 \uhat^3 \}
\]
The $x^3$-torsion generators are:
\[
\Zq[\vhat_1,v_2^{\pm 4}]\{2^{q-1}\vhat_1 \uhat\}
\qquad
\Zq[v_2^{\pm 4},\uhat^2]\{\uhat^4, 2^{q-1}\uhat^5 \}.
\]
The $x^7$-torsion generators are:
$
\Zq[\vhat_2^{\pm 1}]
\{ 2^{q-1} \uhat, \uhat^2 , 2^{q-1}  \uhat^3 \}.
$

We now remove the elements coming from
$\text{im}(N_*^{res})$. 
Modulo powers of $\uhat^2$, we have that
\[
\vhat_1^i v_2^{2j} \uhat^{2k}  v_2 (\uhat - c(\uhat)) \lra 
\vhat_1^i v_2^{2j}  \uhat^{2k} 2 v_2 \uhat.
\]
This is the very first term of the $x^1$-torsion above.

Our next concern is 
\[
\uhat \lra N_*(\uhat) = \xi(\pontr) \lra \uhat + c(\uhat).
\]
By Lemma \ref{lemma}, we have
\[
\uhat + c(\uhat) = \vhat_1 \uhat^2 + \vhat_2 \uhat^4 
\qquad 
\mod (2, \vhat_1^2, \uhat^5).
\]

We have eliminated the first term in our $x^1$-torsion and
now we use the first term above to set $\vhat_1 \uhat^2$ and
$z \vhat_1 \uhat^2 = 2^{q-1} \vhat_1 \uhat^3$ equal to $ 0$.
That reduces our remaining $x^1$-torsion to only:
\[
\Z/(2^{q-1})[v_2^{\pm 2},\uhat^2] \{ 2 \uhat^2 \}
\]
The $x^3$ and $x^7$-torsion remain unaffected since
$\text{im}(N_*^{res})$ is $x^1$-torsion. 
We did not use the term $z v_2 (\uhat - c(\uhat))$, but there is
nothing left in this degree to hit.

To consolidate, the $x^i$-torsion generators of
\[ER(2)^*(B \Zqq)/( \text{im}(N_*^{res}))\]
are described as:
\begin{enumerate}
\item[]
$x^1\mbox{-torsion}\quad 
\Z/(2^{q-1})[v_2^{\pm 2},\uhat^2] \{ 2 \uhat^2 \} $
\item[]
$x^3\text{-torsion}\quad 
\Zq[\vhat_1,v_2^{\pm 4}]\{2^{q-1}\vhat_1 \uhat\}
\quad
\Zq[v_2^{\pm 4},\uhat^2]\{\uhat^4, 2^{q-1}\uhat^5 \} $
\item[]
$x^7\text{-torsion}\quad 
\Zq[\vhat_2^{\pm 1}]
\{ 2^{q-1} \uhat, \uhat^2 , 2^{q-1}  \uhat^3 $\}.
\end{enumerate}

We set $\vhat_2 = 1 = v_2^{-8}$ to simplify our computations.

Having done this, we want to rewrite our answer in terms
of the elements this description represents, namely, in
our associated graded object:
$2v_2^{2i} = \alpha_i$, $\uhat^2 = -\pontr$, $v_2^4 \uhat^4
= \vhat_1 v_2^4 \uhat^2 = w\pontr$, 
$\vhat_1 v_2^4 = w$,
and $z = 2^{q-1} \uhat$.

Now we rewrite 
$ER(2)^*(B \Zqq)/( \text{im}(N_*^{res}))$ as:
\begin{enumerate}
\item[]
$x^1\mbox{-torsion}  \quad 
\Z/(2^{q-1})[\pontr]\{\alpha_i \pontr \} 
\quad 0 \le i < 4$
\item[]
$x^3\text{-torsion}  \quad 
\Zq[\vhat_1]\{\vhat_1 z , w z \}
\quad
\Zq [\pontr]\{\pontr^2, w \pontr, z \pontr^2, z w \pontr \} $
\item[]
$x^7\text{-torsion}  \quad 
\Zq
\{ z , \pontr , z \pontr \}.$
\end{enumerate}

From this description we see that there is a map 
\[
ER(2)^*[[\pontr,z]]
\lra
ER(2)^*(B \Zqq)/( \text{im}(N_*^{res})).
\]
Furthermore, this map must map $(J)$ to zero.
The above analysis 
tells us that the unusual elements 
with the
$\alpha_i $ 
that
must be expressed in higher filtrations can be done so, modulo
the image of $N_*^{res}$, in terms of $ER(2)^*$, $\pontr$, and $z$.
We already knew that this could be done for $2z = \alpha_0 z$, $z^2$, and
$2^q \pontr$. 
This allows us to assert that relations, mod image of $N_*$,
exist like this for $\alpha_{\{1,2,3\}}z$ and $2^{q-1} 
\alpha_{\{1,2,3\}}\pontr$.
These are the relations we use in our Theorem \ref{normseq}.

Taken altogether, we get our map, which is now already obviously surjective:
\[
\frac{ER(2)^*[[\pontr,z]]}{(J)} \lra 
\frac{ER(2)^{*}(B \Zqq)  }{
( \text{im}(N_*^{res}))}  .
\]

All that remains is to show this map is an isomorphism.
We do this by proving an isomorphism on the associated graded objects.
To do that, we analyze the source side of this map.

We begin with $ER(2)^*$ from Fact \ref{ffact}.  We rewrite it as:
\begin{enumerate}
\item[]
$x^1\mbox{-torsion}  \quad 
\Z[\vhat_1]\{\alpha_i\} 
\quad 0 \le i < 4$
\item[]
$x^3\text{-torsion}  \quad 
\Zq[\vhat_1]\{\vhat_1, w \} $
\item[]
$x^7\text{-torsion}  \quad 
\Zq\{1 \}$
\end{enumerate}

We continue with our filtration and associated graded object.
First, we make everything free over this on $\pontr^i$ and
$\pontr^i z$.  We can make $  z^2 = 0$ using our filtration.
Likewise, $2^{q-1} \alpha_i \pontr = 0 = \alpha_i z$.  Writing
this down, we will have used every relation except $\xi(\pontr)=0$.
What we have at this stage is:
\begin{enumerate}
\item[]
$x^1\mbox{-torsion}  \quad 
\Z/(2^{q-1})[\vhat_1, \pontr]\{\alpha_i \pontr\}
\quad 0 \le i < 4$
\item[]
$x^3\text{-torsion}  \quad 
\Zq[\vhat_1,\pontr]\{\vhat_1 \pontr, w  \pontr, \vhat_1 z , w z\} $
\item[]
$x^7\text{-torsion}  \quad 
\Zq[\pontr] \{\pontr, z \}$
\end{enumerate}

Next we take out the first term of $\xi(\pontr)$, i.e. $\vhat_1 \pontr$.
This is a bit different from before
when the   
$\text{im}(N_*^{res})$
was a submodule of $ER(2)^*(B \Zqq)$.
In $(J)$, we are taking out $\xi(\pontr)$ as part of an ideal.
We end up with:
\begin{enumerate}
\item[]
$x^1\mbox{-torsion}  \quad 
\Z/(2^{q-1})[ \pontr]\{\alpha_i \pontr\}
\quad 0 \le i < 4$
\item[]
$x^3\text{-torsion}  \quad 
\Zq[\vhat_1]\{ \vhat_1 z , w z\} 
\quad
\Zq[\pontr]\{ w  \pontr,  w z \pontr\} $
\item[]
$x^7\text{-torsion}  \quad 
\Zq[\pontr] \{\pontr, z \}$
\end{enumerate}

To get this in the same form
as $ER(2)^*(B \Zqq)/ ( \text{im}(N_*^{res})),  $
there is just one last step.
The element $\pontr^2$ that seems
to be $x^7$-torsion is the same, mod higher filtrations, as
$\vhat_1 \pontr$  
(recall $\xi(\pontr)$ starts off as
$\vhat_1 \pontr + \vhat_2 \pontr^2$ mod $(2)$).
This is 
$x^3$-torsion, so we should only
have $\Zq \{z, \pontr, z \pontr \}$ left as $x^7$-torsion and we should
have $\Zq [\pontr]\{\pontr^2, z\pontr^2 \}$ as $x^3$-torsion.  

This shows we have an isomorphism of associated graded objects
and completes the proof.

\end{proof}

\section{Appendix}
\label{appendix}

Because we use this table all the time, it should be available for
general reference.  Here is $ER(2)^*$, written in its 48-periodic
form, with $k \ge 0$ and $\alpha^0 = 1$.

\[
\begin{array}{llllllllll}
15 & x \alpha^{3k+2} &  31 &  x \alpha^{3k}    & 47 & x \alpha^{3k+1}  \\
14 & x^2 \alpha^{3k} &  30 &  x^2 \alpha^{3k+1}    & 46 & x^2 \alpha^{3k+2}  \\
13 &  &  29 &    & 45 & x^3   \\
12 & \alpha_3 \alpha^{3k}  &  28 & x^4, \   \alpha_3 \alpha^{3k+1}  & 44 & 
\alpha_3 \alpha^{3k+2}   \\
11 & x^5 &  27 &     & 43 &    \\
10 &  &  26 &     & 42 & x^6   \\
9 &  &  25 &     & 41 &    \\
8 & w \alpha^{3k+1} &   
24 & \alpha_2, \   \ w\alpha^{3k+2}   &   
40 & w \alpha^{3k}    \\
  &  2w\alpha^{3k+1} = \alpha_2 \alpha^{3k+2}&  
&     2w \alpha^{3k+2} =  \alpha_2 \alpha^{3k+3}  & 
&   2w\alpha^{3k} = \alpha_2 \alpha^{3k+1}  \\
7 & xw \alpha^{3k+2} &  23 & xw \alpha^{3k} & 39 & xw \alpha^{3k+1}  \\
6 & x^2 w \alpha^{3k} &  22 & x^2 w \alpha^{3k+1} & 38 & x^2 w \alpha^{3k+2} \\
5 & &  21 &  & 37 & \\
4 & \alpha_1 \alpha^{3k+1}  &  20 & \alpha_1 \alpha^{3k+2}  & 36 & 
\alpha_1 \alpha^{3k} \\
3 & &  19 &  & 35 &  \\
2 & &  18 &  & 34 &  \\
1 & &  17 &  & 33 &  \\
0 & \alpha^{3k} &  16 & \alpha^{3k+1}  & 32 & \alpha^{3k+2}  \\
\end{array}
\]


\end{document}